\documentclass[a4paper,10pt]{article}

\usepackage{amsthm}
\usepackage{amsmath}
\usepackage{amssymb}
\usepackage{latexsym}
\usepackage[charter]{mathdesign}
\usepackage{eucal}
\usepackage{color}
\usepackage{graphicx,epstopdf}

\usepackage{floatrow}
\newfloatcommand{capbtabbox}{table}[][\FBwidth]

\newtheorem{theorem}{Theorem}[section]

\newtheorem{lemma}[theorem]{Lemma}
\newtheorem{corollary}[theorem]{Corollary}
\newtheorem{remark}[theorem]{Remark}

\newtheorem{example}[theorem]{Example}
\newtheorem{assumption}[theorem]{Assumption}

\begin{document}

\title{On fixed gain recursive estimators with discontinuity in the parameters}

\author{Huy N. Chau\thanks{MTA Alfr\'ed R\'enyi Institute of
Mathematics, Budapest.
E-mail: \texttt{chau@renyi.mta.hu}}
\and Chaman Kumar\thanks{Indian Institute of Technology, Roorkee. E-mail: \texttt{c.kumarfma@iitr.ac.in}} \and
Mikl\'os R\'asonyi\thanks{MTA Alfr\'ed R\'enyi Institute of
Mathematics, Budapest.
E-mail: \texttt{rasonyi@renyi.mta.hu}}  
\and Sotirios Sabanis\thanks{University of Edinburgh. E-mail: \texttt{S.Sabanis@ed.ac.uk}}}

\date{\today}

\maketitle

\begin{abstract}
In this paper we estimate the expected tracking error of a fixed gain stochastic approximation scheme.
The underlying process is not assumed Markovian, a mixing condition is required instead. Furthermore,
the updating function may be discontinuous in the parameter.
\end{abstract}



\section{Introduction}

Let $\mathbb{N}:=\{0,1,2,\ldots\}$. We are interested in stochastic approximation procedures where
a parameter estimate ${\theta}_t$, $t\in\mathbb{N}$ is updated by a recursion of the form
\begin{equation}\label{elso}
{\theta}_{t+1}={\theta}_t+\gamma_{t+1} H({\theta}_t,X_{t+1}),\ t\in\mathbb{N},
\end{equation}
starting from some guess $\theta_0$. Here $X_t$ is a stationary signal, $\gamma_t$ is a sequence of real numbers and
$H(\cdot,\cdot)$ is a given functional. The most common choices are
$\gamma_t=1/t$ (decreasing gain) and $\gamma_t:=\lambda$ (fixed gain).
The former family of procedures is aimed to converge to $\theta^*$ with $G(\theta^*)=0$ where
$G(\theta):=EH(\theta,X_t)$.
The latter type of procedures is supposed to ``track'' $\theta^*$, even when the system dynamics is (slowly)
changing.

In most of the related literature the error analysis of \eqref{elso} was carried out
only in the case where $H$ is (Lipschitz-)continuous in $\theta$. This restrictive hypothesis fails to accommodate discontinuous procedures which are common in practice, e.g.
the signed regressor, signed error and sign-sign algorithms (see \cite{buck}, \cite{e}, \cite{e2})
or the Kohonen algorithm (see \cite{kohonen1, kohonen}). 
Recently, the decreasing
gain case was investigated in \cite{fort} for controlled Markov chains and the procedure \eqref{elso} was shown to
converge almost surely under appropriate
assumptions, without requiring continuity of $H$. We refer to \cite{fort} for a review of the relevant literature and for
examples.

The purpose of the present article is an exploration of the case where $X_t$ has possibly non-Markovian dynamics.
We consider fixed gain procedures and weaken continuity of $H$ to continuity in the sense
of conditional expectations, see \eqref{kac} below, compare also to condition $\mathbf{H4}$ in \cite{fort}.

We follow the methodology of the papers \cite{laci8,laci2,laci3} which are based on the concept
of $L$-mixing, coming from \cite{laci1}. Our arguments work under a modification of the
original definition of $L$-mixing, see Section \ref{ketto}. We furthermore assume a certain asymptotic forgetting
property, see Assumption \ref{infsum}. We manage
to estimate the tracking error for \eqref{elso}, see our main result, Theorem \ref{main} in Section \ref{harom}.

At this point we would like to make comparisons with another important reference, \cite{lp},
where no Markovian or continuity assumptions were made, certain averaging properties
of the driving process were required instead. It follows from Subsection 4.2 of \cite{lp} that almost sure convergence 
of a decreasing gain procedure can be guaranteed under the $\alpha$-mixing property of the driving process,
see e.g. \cite{ethier} about various mixing concepts. It seems that establishing the $L$-mixing property is often 
relatively simple while $\alpha$-mixing is rather stringent and difficult to prove.
In addition, our present work provides explicit estimates for the error.
See Section \ref{negy} for examples 
illustrating the scope of Theorem \ref{main}.

Section \ref{numera} reports simulations showing that the theoretical estimate
is in accordance with numerical results. 
Proofs for Sections \ref{ketto} and \ref{harom} are relegated to Section \ref{appi}.

\section{$L$-mixing and conditional $L$-mixing}\label{ketto}

Estimates for the error of stochastic approximation schemes like \eqref{elso}
can be proved under various ergodicity assumptions
on the driving process. It is demonstrated in \cite{laci2} and \cite{laci3} that the concept of $L$-mixing (see its 
definition below in the present section) is
sufficiently strong for this purpose. An appealing feature of $L$-mixing is that it can easily be
applied in non-Markovian contexts as well, see Section \ref{negy}.

It turns out, however, that for discontinuous updating functions $H$ the arguments of \cite{laci2,laci3} break down.
To tackle discontinuities, we introduce a new concept of mixing here, which is of interest on its own right.

Throughout this paper we are working on a probability space $(\Omega,\mathcal{F},P)$ that is equipped
with a discrete-time filtration $\mathcal{F}_n$, $n\in\mathbb{N}$ as well as with a
decreasing sequence of sigma-fields $\mathcal{F}_n^+$, $n\in\mathbb{N}$ such that $\mathcal{F}_n$ is
independent of $\mathcal{F}_n^+$, for all $n$.

Expectation of
a random variable $X$ will be denoted by $EX$.
For any $m\geq 1$, for any $\mathbb{R}^m$-valued random variable $X$ and for any $1\leq p<\infty$, let us set
$\Vert X\Vert_p:=\sqrt[p]{E|X|^p}$. We denote by $L^p$ the set of $X$ satisfying $\Vert X\Vert_p<\infty$.
The indicator function of a set $A$ will be denoted by $1_A$.

We now present the class of $L$-mixing processes which were introduced in \cite{laci1}. This concept proved to
be extremely useful in solving certain hard problems of system identification, see e.g. \cite{laci4,laci5,laci6,laci7,q}.

Fix an integer $N\geq 1$ and let $D\subset \mathbb{R}^N$ be a set of parameters. A measurable function 
$X:\mathbb{N}\times D\times\Omega\to\mathbb{R}^m$ is called a random field. We will drop dependence on $\omega\in\Omega$ and
use the notation $X_t(\theta)$, $t\in\mathbb{N}$, $\theta\in D$. 

For any $r\geq 1$, a
random field $X_t(\theta)$, $t\in\mathbb{N}$,
$\theta\in D$ is called \emph{bounded in $L^r$} if
\begin{equation}\label{rysy}
M_r(X):=\sup_{\theta\in D}\sup_{t\in\mathbb{N}}\Vert X_t(\theta)\Vert_r<\infty.
\end{equation}
For an $L^r$-bounded $X_t(\theta)$, define also the quantities
$$
\gamma_r(\tau,X):=\sup_{\theta\in D}\sup_{t\geq\tau} \Vert X_t(\theta)-E[X_t(\theta)\vert \mathcal{F}_{t-\tau}^+]\Vert_r,
\ \tau\geq 1,
$$
and
\begin{equation}\label{mysy}
\Gamma_r(X):=\sum_{\tau=1}^{\infty}\gamma_r(\tau,X).
\end{equation}

For some $r\geq 1$, a random field $X_t(\theta)$ is called \emph{uniformly $L$-mixing of order} $r$ (ULM-$r$) if it is 
bounded in $L^r$; for all $\theta\in D$, $X_t(\theta)$, $t\in\mathbb{N}$ is
adapted to $\mathcal{F}_t$, $t\in\mathbb{N}$;  and
$\Gamma_r(X)<\infty$. Here uniformity refers to the parameter $\theta$.
Furthermore, $X_t(\theta)$ is called \emph{uniformly $L$-mixing} if it is uniformly $L$-mixing of order $r$ for all $r\geq 1$.

In the case of a single stochastic process (which corresponds to the
case where the parameter set $D$ is a singleton) we apply the terminology ``$L$-mixing process of order $r$'' and
``$L$-mixing process''.

\begin{remark}\label{kony}{\rm The $L$-mixing property shows remarkable stability under various operations,
this is why it proved to be a versatile tool in the analysis of stochastic systems, see
\cite{laci2,laci3,laci4,laci5,laci6,laci7,q}.
If $F$ is a Lipschitz function and $X_t(\theta)$ is ULM-$r$ then $F(X_t(\theta))$ is also
ULM-$r$, by \eqref{dirk} in Lemma \ref{mall} below.
Actually, if $F$ is such that $|F(x)-F(y)|\leq K(1+|x|^k+|y|^k)|x-y|$ for all $x,y\in\mathbb{R}$ with some $k,K>0$ 
then $F(X_t(\theta))$ is uniformly $L$-mixing whenever $X_t(\theta)$ is,
see Proposition 2.4 of \cite{q}.
Stable linear filters also preserve the $L$-mixing property, see \cite{laci1}. Proving that $F(X_t(\theta))$
is $L$-mixing for discontinuous $F$ is more delicate, see Section \ref{negy} for helpful techniques.

Other mixing conditions could alternatively be used. Some of these 
are inherited by arbitrary measurable functions of the respective
processes (e.g. $\phi$-mixing, see Section 7.2 of \cite{ethier}). However, they are considerably 
difficult to verify while $L$-mixing (and its conditional version to be defined below)
is relatively simple to check, see also the related remarks on page 2129 of \cite{laci7}.}
\end{remark}

Recall that, for any family $Z_i$, $i\in I$ of real-valued random variables, $\mathrm{ess.}\sup_{i\in I} Z_i$
denotes a random variable that is an almost sure upper bound for each $Z_i$ and it is a.s.
smaller than or equal to any other such bound. Such an object is known to exist, independently of
the cardinality of $I$, and it is a.s. unique, see e.g. Proposition VI.1.1. of \cite{neveu}.

Now we define conditional $L$-mixing, inspired by \eqref{rysy} and \eqref{mysy}.
Let $X_t(\theta)$, $t\in\mathbb{N}$, $\theta\in D$ be a random field bounded in $L^r$ for some $r\geq 1$
and define, for each $n\in\mathbb{N}$,
\begin{eqnarray*}
	M^{n}_r(X) &:=& \mathrm{ess}\sup_{\theta\in D}\sup_{t \in\mathbb{N}} 
	E^{1/r}[|X_{n+t}(\theta)|^r\big\vert\mathcal{F}_n],\\
	\gamma^{n}_r(\tau,X)&:=& \mathrm{ess}\sup_{\theta\in D}\sup_{t\geq\tau} 
	E^{1/r}[|X_{n+t}(\theta)-E[X_{n+t}(\theta)\vert \mathcal{F}_{n+t-\tau}^+\vee \mathcal{F}_n]|^r\big\vert
	\mathcal{F}_n],\ \tau\geq 1,\\
	\Gamma^{n}_r(X) &:=&\sum_{\tau= 1}^{\infty}\gamma^{n}_r(\tau,X).
\end{eqnarray*}

For some $s,r\geq 1$, we call $X_t(\theta)$, $t\in\mathbb{N}$, $\theta\in D$
uniformly \emph{conditionally} $L$-mixing of order $(r,s)$ (abbreviation: UCLM-$(r,s)$) if 
it is $L^r$-bounded; $X_t(\theta)$, $t\in\mathbb{N}$ is adapted to 
$\mathcal{F}_t$, $t\in\mathbb{N}$ 
for all $\theta\in D$ 
and the sequences  $M^n_r(X)$, $\Gamma^n_r(X)$, $n\in\mathbb{N}$
are bounded in $L^s$. When the UCLM-$(r,s)$ property holds for all $r,s\geq 1$ then we simply say that the
random field is uniformly conditionally $L$-mixing.
In the case of stochastic processes (when $D$ is a singleton)
the terminology ``conditionally $L$-mixing process of order $(r,s)$'' (respectively, conditionally
$L$-mixing process) will be used. 

\begin{remark}\label{aci}
{\rm Note that if $\mathcal{F}_0$ is trivial and $X_t(\theta)$ is UCLM-$(r,1)$ then it is also ULM-$r$. Indeed,
in that case
\begin{eqnarray*}
M_r(X) = M_r^{0}(X),\quad
\Gamma_r(X) = \Gamma_r^{0}(X).
\end{eqnarray*}
For non-trivial $\mathcal{F}_0$, however, no such implication holds.}  
\end{remark}

\begin{remark}\label{kony1}{\rm If $F$ is a Lipschitz function and $X_t(\theta)$ is UCLM-$(r,1)$ then $F(X_t(\theta))$ is also
UCLM-$(r,1)$, by Lemma \ref{mall} below. Conditional versions of the arguments in Lemma \ref{product} show that
if $X_t(\theta)$ is UCLM-$(rp,1)$ and $Y_t(\theta)$ is UCLM-$(rq,1)$ (where $1/p+1/q=1$)
then 
\begin{eqnarray}\label{betlehem1}
M^n_r(XY)&\leq& M^n_{rp}(X)M^n_{rq}(Y),\\ 
\Gamma^n_r(XY) &\leq& 2M^n_{rp}(X)\Gamma^n_{rq}(Y) +2\Gamma^n_{rp}(X)M^n_{rq}(Y).\label{betlehem2}
\end{eqnarray}}
\end{remark}

We now present another concept, a surrogate for continuity in $\theta\in D$.
We say that the random field $X_t(\theta)\in L^1$, $t\in\mathbb{N}$, $\theta\in D$ satisfies
the \emph{conditional Lipschitz-continuity} (CLC) property if there is a deterministic $K>0$ such that,
for all $\theta_1,\theta_2\in D$ and for all $n\in\mathbb{N}$,
\begin{equation}\label{kac}
E\left[ \vert X_{n+1}(\theta_1)-X_{n+1}(\theta_2)\vert\, \Big|\mathcal{F}_n\right]\leq K|\theta_1-\theta_2|,\mbox{ a.s.}
\end{equation}

Pathwise discontinuities of $\theta\to X_n(\theta)$ can often be smoothed out and \eqref{kac} can be
verified by imposing some conditions on the one-step 
conditional distribution of $X_{n+1}$ given $\mathcal{F}_n$, see Assumption \ref{smooth} and Lemma \ref{lem_CLC} below.

\begin{remark}{\rm We comment on the differences between condition $\mathbf{H4}$ of \cite{fort} and our CLC property. Assume that $X$ is stationary and Markovian. 
On one hand, $\mathbf{H4}$ of \cite{fort} stipulates that, for $\delta>0$ 
\begin{equation}\label{ooo}
\sup_{\theta\in D_c}E\left[\sup_{\theta'\in D_c,\ |\theta-\theta'|\leq \delta}|H(\theta,X_1)-H(\theta',X_1)|
\right]\leq K\delta^{\alpha}
\end{equation}
for any compact $D_c\subset D$ with some $K>0$ (that may depend on $D_c$) and with some $0<\alpha\leq 1$
(independent of $D_c$). On the other hand, CLC is equivalent to
\begin{equation}\label{pok}
\sup_{\theta,\theta'\in D,\ |\theta-\theta'|\leq \delta}
E\left[\left|H(\theta,X_1)-H(\theta',X_1)\right|\,\big\vert X_0=x\right]\leq K\delta. 
\end{equation}
for $\mathrm{Law}(X_0)$-almost every $x$.
Clearly, \eqref{ooo} allows H\"older-continuity (i.e. $\alpha<1$) while \eqref{pok} 
requires Lipschitz-continuity. In the case $\alpha=1$ \eqref{ooo} is not
comparable to CLC though both express a kind of ``continuity in the average''.}
\end{remark}


The main results of our paper require a specific structure for the sigma-algebras which facilitates 
to deduce properties of conditional $L$-mixing processes from those of ``unconditional'' ones. More precisely,
we rely on the crucial Doob-type inequality in Theorem \ref{estim} below. This could probably be proved for arbitrary
sigma-algebras but only at the price of redoing all the tricky arguments of \cite{laci1}
in a more difficult context. We refrain from this since Theorem \ref{estim} can accommodate most models
of practical importance. Let $\mathbb{Z}$ denote the set of integers.

\begin{theorem}\label{estim} 
Fix $r>2$, $n\in \mathbb{N}$. Assume that, for all $t\in\mathbb{N}$, $\mathcal{F}_t=\sigma(\varepsilon_j,\ j\in\mathbb{N},\ j\leq t)$, $\mathcal{F}_t^+:=\sigma(\varepsilon_j,\ j>t)$
for some i.i.d. sequence $\varepsilon_j$, $j\in\mathbb{Z}$ with values in some Polish space $\mathcal{X}$.
Let $W_t$, $t\in\mathbb{N}$ be a conditionally $L$-mixing process of order $(r,1)$, satisfying 
$E[W_t\vert\mathcal{F}_n]=0$ a.s. for all {$t\geq n$}. 
Let $m >n$ and let $b_t$, $n< t\leq m$ be deterministic numbers. Then we have
\begin{equation}\label{mandrill}
E^{1/r}\left[ \max_{n < t \le m} \left| \sum_{s = n+1}^{t} b_s W_s \right|^r \big\vert\mathcal{F}_n \right]
 \le C_r \left( \sum_{s=n+1}^{m} b_s^2 \right)^{1/2} \sqrt{{M}_r^{n}(W) \Gamma_r^{n}(W)},
\end{equation}
almost surely, where $C_r$ is a deterministic constant depending only on $r$ but independent of $n,m$.
\end{theorem}

The proof is reported in Section \ref{appi}.

\section{Fixed gain stochastic approximation}\label{harom}

Let $N \ge 1$ be an integer
and let $\mathbb{R}^N$ be the Euclidean space with norm $\vert x\vert:=\sqrt{\sum_{i=1}^N x_i^2}$, $x\in\mathbb{R}^N$. 
Let $D\subset\mathbb{R}^N$ be a bounded (nonempty) open set representing possible system parameters.
Let $H:D\times\mathbb{R}^m\to \mathbb{R}^N$ be a bounded measurable function.
We assume throughout this section that for all $t\in\mathbb{N}$, $\mathcal{F}_t=\sigma(\varepsilon_j,\ j\in\mathbb{N},\ j\leq t)$, $\mathcal{F}_t^+:=\sigma(\varepsilon_j,\ j>t)$
for some i.i.d. sequence $\varepsilon_j$, $j\in\mathbb{Z}$ with values in some Polish space $\mathcal{X}$, 
in particular the condition on the sigma algebras in the statement of Theorem \ref{estim} holds.

Let 
\begin{equation}\label{oi}
X_t:=g(\varepsilon_t,\varepsilon_{t-1},\ldots),\ t\in\mathbb{N},
\end{equation}
with some fixed measurable function $g:\mathcal{X}^{-\mathbb{N}}\to\mathbb{R}^m$. 
Clearly, $X$ is a (strongly) stationary $\mathbb{R}^m$-valued process, see Lemma 10.1 of 
\cite{kallenberg}.

\begin{remark}\label{talann}
{\rm We remark that, in the present setting, the CLC property holds if, for all $\theta_1,\theta_2\in D$,
\begin{equation*}\label{kacc}
E\left[ \vert H(\theta_1, X_{1})-H(\theta_2,X_1)\vert\, \Big|\mathcal{F}_0\right]\leq K|\theta_1-\theta_2|,\mbox{ a.s.},
\end{equation*}
due to the fact that the law of $(X_{k+1},\varepsilon_{k},\varepsilon_{k-1},\ldots)$
is the same as that of $(X_{1},\varepsilon_{0},\varepsilon_{-1},\ldots)$, for all $k\in\mathbb{Z}$.}
\end{remark}



Define $G(\theta):=EH(\theta,X_0)$. Note that, by stationarity of $X$, $G(\theta)=EH(\theta,X_t)$ for all
$t\in\mathbb{N}$. We need some stability hypotheses formulated in terms of an ordinary differential equation related to $G$.

\begin{assumption}\label{boun}
On $D$, the function $G$ is twice continuously differentiable and bounded, together with its first and second derivatives.
\end{assumption}

Fix $\lambda>0$. Under Assumption \ref{boun}, the equation
\begin{equation}\label{ode}
\dot{y}_t=\lambda G(y_t),\quad y_s=\xi,
\end{equation}
has a unique solution for each $s\geq 0$ and $\xi\in D$, on some (finite or infinite) interval
$[s,v(s,\xi))$ with $v(s,\xi)>s$. We will denote this solution by $y(t,s,\xi)$, $t\in [s,v(s,\xi))$.
Let $D_1 \subset D$ such that for all $\xi \in D_1$ we have $y(t,0,\xi) \in D$
for any $t \ge 0$.
We denote
$$\phi(D_1) = \{ u\in D :\, u=y(t,0,\xi), \text{ for some }t \ge 0,\ \xi \in D_1  \}.$$
The $\varepsilon$-neighbourhood of a set $D_1$ is denoted by $S(D_1,\varepsilon)$, i.e.
$$S(D_1,\varepsilon) = \{ u\in\mathbb{R}^N: |u - \theta| < \varepsilon \text{ for some } \theta \in D_1 \}.$$
We remark that, under Assumption \ref{boun}, the function $y(t,s,\xi)$ is continuously differentiable in $\xi$.

Notice that all the above observations would be true under weaker hypotheses than those of Assumption \ref{boun}.
However, the proof of Lemma \ref{expstab} below requires the full force of Assumption \ref{boun}, see \cite{laci3}.

\begin{assumption}\label{stab}
There exist open sets
$$
\emptyset\neq D_{\xi} \subset D_y \subset D_{\theta} \subset D_{\overline{y}} \subset D
$$
such that $\phi(D_{\xi}) \subset D_y$, $S(D_y,d) \subset D_{\theta}$ for some $d>0$ and
$\phi(D_{\theta})  \subset D_{\overline{y}}$, $S(D_{\overline{y}},d') \subset D$ for some $d'>0$.
The ordinary differential equation (\ref{ode}) is exponentially asymptotically stable with respect to initial perturbations, i.e.
there exist $C^* >0, \alpha >0$ such that, for each $\lambda$ sufficiently small, for all $0\leq s\leq t$, $\xi \in D$
\begin{equation}\label{gag}
\left| \frac{\partial}{\partial{\xi}} y(t,s,\xi) \right| \leq C^* e^{-\lambda \alpha(t-s)}.
\end{equation}
We furthermore assume that there is $\theta^*\in D$ such that
\begin{equation}\label{g}
G(\theta^*)=0.
\end{equation}
\end{assumption}

It follows from $\phi(D_{\xi})\subset D_y$ and \eqref{gag} that $\theta^*$ actually lies in the closure of $D_y$
and that there is only one $\theta^*$ satisfying \eqref{g}.

While Assumptions \ref{boun}, \ref{stab} pertained to a deterministic equation, our next hypothesis is
of a stochastic nature. 

\begin{assumption}\label{infsum}
For all $n\in\mathbb{N}$,
\begin{equation}\label{num}
E\left[\sup_{\vartheta\in D}\sum_{k=n}^{\infty} \left| E[H(\vartheta,X_{k+1})\vert\mathcal{F}_{n}] -G(\vartheta) \right|\right]<\infty
\end{equation}
\end{assumption}

\begin{remark}\label{frater}
{\rm Assumption \ref{infsum} expresses a certain kind of ``forgetting'': for $k$ large,
$E[H(\vartheta,X_{k+1})\vert\mathcal{F}_{n}]$ is close to  $G(\vartheta)=EH(\theta,X_{k+1})\vert_{\theta=\vartheta}$ in $L^1$, uniformly in $\vartheta$ and
the convergence is fast enough so that the sum in \eqref{num} is finite. In other words,
this is again a kind of mixing property.

In certain cases, the validity of Assumption \ref{infsum} indeed follows from $L$-mixing.
Let $X_t$, $t\in\mathbb{N}$ be $L$-mixing of order $1$  and let $x\to H(\theta,x)$
be Lipschitz-continuous with a Lipschitz constant $L^{\dagger}$ that is independent of $\theta$. 
We claim that Assumption \ref{infsum} holds under these conditions.
Indeed, for every $\vartheta\in D$
\begin{eqnarray*}
\sum_{k=n}^{\infty} E\left\vert E\left[ H(\vartheta,X_{k+1}) \vert\mathcal{F}_{n}\right]-
E[H(\theta,X_{k+1})]\vert_{\theta=\vartheta} \right\vert &\leq&\\
\sum_{k=n}^{\infty} \left\vert 
E\left[ H(\vartheta,X_{k+1}) \vert\mathcal{F}_{n}\right]-
E\left[ H(\vartheta,E[X_{k+1}\vert\mathcal{F}^+_{n}]) \vert\mathcal{F}_{n}\right] 
\right\vert &+&\\
\sum_{k=n}^{\infty} \left\vert E\left[ H(\theta,E[X_{k+1}\vert\mathcal{F}^+_{n}])\right]\vert_{\theta=\vartheta} -
E[H(\theta,X_{k+1})]\vert_{\theta=\vartheta} \right\vert &\leq&\\
2L^{\dagger} \sum_{k=n}^{\infty} \left\vert X_{k+1}- E[X_{k+1}\vert\mathcal{F}^+_{n}]\right\vert 
\end{eqnarray*}
noting that 
$$
E\left[ H(\vartheta,E[X_{k+1}\vert\mathcal{F}^+_{n}]) \vert\mathcal{F}_{n}\right]=
E\left[ H(\theta,E[X_{k+1}\vert\mathcal{F}^+_{n}])\right]\vert_{\theta=\vartheta},
$$
by independence of $\mathcal{F}_{n}$ and $\mathcal{F}_{n}^+$. Hence
$$
E\left[\sup_{\vartheta\in D}\sum_{k=n}^{\infty} \left| E[H(\vartheta,X_{k+1})\vert\mathcal{F}_{n}] -G(\vartheta) \right|\right]\leq 2L^{\dagger} \Gamma_1(X)<\infty.
$$
Assumption \ref{infsum} can also 
be verified in certain cases where $H$ is discontinuous, see Section \ref{negy}.}
\end{remark}

We now state the main result of our article.

\begin{theorem}\label{main}
Let $H(\theta,X_t)$ be UCLM-$(r,1)$ for some $r>2$, satisfying the CLC property (see \eqref{kac} above).
Let Assumptions \ref{boun}, \ref{stab}
and \ref{infsum} be in force. For some $\xi\in D_{\xi}$, define the recursive procedure
\begin{equation}\label{bab}
\theta_0:=\xi,\quad {\theta}_{t+1}={\theta}_t+\lambda H({\theta}_t,X_{t+1}),
\end{equation}
with some $\lambda>0$. Define also its ``averaged'' version,
\begin{equation}\label{averaged}
z_0:=\xi,\quad z_{t+1}=z_t+\lambda G(z_t).
\end{equation}
Let $d,d'$ in Assumption \ref{stab} be large enough and let $\lambda$ be small enough.
Then  $\theta_t,z_t\in D_{\theta}$ for all $t$ and there is a constant $C$,
independent of $t\in\mathbb{N}$ and
of $\lambda$, such that
$$
E\left\vert {\theta}_t-z_t\right\vert\leq C\lambda^{1/2},\ t\in\mathbb{N}.
$$
\end{theorem}

An important consequence of the main theorem is provided as follows.

\begin{corollary}\label{karszt}
Under the conditions of Theorem \ref{main}, there is $t_0(\lambda)\in\mathbb{N}$ such that
$$
E\left\vert {\theta}_t-\theta^*\right\vert\leq C\lambda^{1/2},\ t\geq t_0(\lambda).
$$
Furthermore, $t_0(\lambda)\leq C^{\circ}\ln(1/\lambda)/\lambda$ for some $C^{\circ}>0$ .
\end{corollary}

The proofs of Theorem \ref{main} and Corollary \ref{karszt} are postponed to Section \ref{appi}.

\begin{remark}{\rm
Our current investigations were motivated by \cite{laci3} where not only the random field $H(\theta,X_t)$
was assumed $L$-mixing,
but also its ``derivative field'' 
\begin{equation}\label{dervi}
\frac{H(\theta_1,X_t)-H(\theta_2,X_t)}{\theta_1-\theta_2},\ t\in\mathbb{N},\ \theta_1,\theta_2\in D,\
\theta_1\neq\theta_2.
\end{equation}
As shown in Section 3 of \cite{laci1}, the latter hypothesis necessarily implies the continuity (in $\theta$)
of $H(\theta,X_t)$. For our purposes such an assumption is thus too strong. We are
able to drop continuity at the price of modifying the $L$-mixing concept, as explained in Section \ref{ketto} above.

We point out that our results complement those of \cite{laci3} even in the case where $H$ is Lipschitz-continuous (in that case the CLC property of our paper obviously holds). 
In \cite{laci3}, the derivative field \eqref{dervi} was assumed to be $L$-mixing. In the present paper
we do not need this hypothesis (but we assume conditional $L$-mixing of order $(r,1)$ for some $r>2$ instead of $L$-mixing).}
\end{remark}

\section{Examples}\label{negy}

The present section serves to illustrate the power of Theorem \ref{main} above by exhibiting processes $X_t$ and
functions $H$ to which that theorem applies.

The (conditional) $L$-mixing property can be verified
for arbitrary bounded measurable functionals of Markov processes with the Doeblin condition (see \cite{gm}) and 
this could probably be extended to a larger family of Markov processes using ideas of \cite{bw} or \cite{hm}. 
We prefer not to review the corresponding methods here but to present some non-Markovian examples because they
demonstrate better the advantages of our approach over the existing literature.

In Subsection \ref{lini} linear processes (see e.g. Subsection 3.2 of \cite{giraitis}) with polynomial
autocorrelation decay are considered, while Subsection \ref{mark} presents a class of
Markov chains in a random environment with contractive properties.

\subsection{Causal linear processes}\label{lini}

\begin{assumption}\label{linear} 
Let $\varepsilon_j$, $j\in\mathbb{Z}$ be a sequence of independent, identically distributed real-valued
random variables such that $E|\varepsilon_0|^{\zeta}<\infty$ for some $\zeta\geq 2$ and $E\varepsilon_0=0$. 
We set 
$\mathcal{F}_n = \sigma(\varepsilon_i, i\le n)$, and $\mathcal{F}^+_n = \sigma(\varepsilon_i, i>n)$ for each $n \in \mathbb{Z}$. Let
us define the process
\begin{equation}\label{majk}
X_t:=\sum_{j=0}^{\infty} a_j \varepsilon_{t-j},\quad t\in\mathbb{Z},
\end{equation}
where $a_j\in\mathbb{R}$, $j\in\mathbb{N}$. We assume $a_0\neq 0$ and
$$
|a_j|\leq C_1 (j+1)^{-\beta},\ j\in\mathbb{N},
$$
for some constants $C_1>0$ and $\beta>1/2$.
\end{assumption}

Note that the series \eqref{majk} converges a.s. (by Kolmogorov's theorem, see e.g. Chapter 4
of \cite{kallenberg}). 
As a warm-up, we now check the
conditional $L$-mixing property for $X$.

\begin{lemma}\label{x}
Let Assumption \ref{linear} be in force. If $\beta>3/2$ 
then the process $X_t$, $t\in\mathbb{N}$ is conditionally $L$-mixing of order $(\zeta,1)$.
\end{lemma}
\begin{proof} We have, for 
$t\in\mathbb{N}$,
\begin{eqnarray*}
E^{1/\zeta}[|X_t|^{\zeta}\vert\mathcal{F}_0] &\leq&
E^{1/\zeta}\left[ 2^{\zeta-1}|\sum_{j=0}^{t-1}  a_j\varepsilon_{t-j}|^{\zeta}\big\vert\mathcal{F}_0\right]+
E^{1/\zeta}\left[ 2^{\zeta-1}|\sum_{j=t}^{\infty}  a_j\varepsilon_{t-j}|^{\zeta} \big\vert\mathcal{F}_0\right]\\ &\leq&
2^{\frac{\zeta-1}{\zeta}}\sum_{j=0}^{t-1}  |a_j|\Vert\varepsilon_{t-j}\Vert_{\zeta}+
2^{\frac{\zeta-1}{\zeta}}\sum_{j=t}^{\infty}  |a_j\varepsilon_{t-j}| 
\end{eqnarray*}
using the simple inequality $(x+y)^{\zeta}\leq 2^{\zeta-1}(x^{\zeta}+y^{\zeta})$, $x,y\geq 0$;
properties of the norm $\Vert\cdot\Vert_{\zeta}$; independence of $\varepsilon_j$, $j\geq 1$ from $\mathcal{F}_0$
and $\mathcal{F}_0$-measurability of $\varepsilon_j$, $j\leq 0$. Hence
\begin{eqnarray*}
E^{1/\zeta}[|X_t|^{\zeta}\vert\mathcal{F}_0] &\leq&
2^{\frac{\zeta-1}{\zeta}}\Vert\varepsilon_{0}\Vert_{\zeta}\sum_{j=0}^{\infty}  |a_j| +
2^{\frac{\zeta-1}{\zeta}}\sum_{j=0}^{\infty} C_1(t+j+1)^{-\beta}|\varepsilon_{-j}|\\ &\leq& 
2^{\frac{\zeta-1}{\zeta}}\Vert\varepsilon_{0}\Vert_{\zeta}\sum_{j=0}^{\infty}  |a_j| +
2^{\frac{\zeta-1}{\zeta}}\sum_{j=0}^{\infty} C_1(j+1)^{-\beta}|\varepsilon_{-j}|\\
&\leq& C_2 \left[1+\sum_{j=0}^{\infty} |\varepsilon_{-j}|(j+1)^{-\beta}\right],
\end{eqnarray*}
for some $C_2>0$. Note that the latter bound is independent of $t$.
Similar estimates prove that, for all $n\geq 0$, 
$$
M^n_{\zeta}(X)\leq C_2 \left[1+\sum_{j=0}^{\infty} |\varepsilon_{n-j}|(j+1)^{-\beta}\right].
$$
The right-hand side has the same law for all $n$ and it is in $L^1$ since $\beta>1$.
This implies that the sequence $M^n_{\zeta}(X)$, $n\in\mathbb{N}$ is bounded 
in $L^{1}$.

For $1\leq m$ and for any $t\in\mathbb{Z}$,
define $$
X_{t,m}^+:=\sum_{j=0}^{m-1} a_j\varepsilon_{t-j},
$$
and, for $t\geq m$, let
$$
X^{\circ}_{t,m}:=X_{t,m}^+ +\sum_{j=t}^{\infty} a_j\varepsilon_{t-j}.
$$
Notice that $E[X_t\vert\mathcal{F}_{t-m}^+\vee\mathcal{F}_{0}]=X_{t,m}^{\circ}$ and, by independence
of $\varepsilon_j$, $j\geq 1$ from $\mathcal{F}_0$,
\begin{eqnarray}\label{ghjk}
E[|X_t-X_{t,m}^{\circ}|^{\zeta} \big\vert\mathcal{F}_0]=\left\Vert\sum_{j=m}^{t-1} a_j\varepsilon_{t-j}\right\Vert_{\zeta}^{\zeta}
&\leq&\\ \nonumber
C_3 E\left(\sum_{j=m}^{t-1} a^2_j\varepsilon^2_{t-j}\right)^{\zeta/2}
&\leq&\\ \nonumber
C_3\left(\sum_{j=m}^{t-1} \Vert a^2_j\varepsilon^2_{t-j}\Vert_{\zeta/2}
\right)^{\zeta/2}
&\leq&\\ \nonumber
C_3\left(C_1^2 \sum_{j=m}^{\infty} (1+j)^{-2\beta} \Vert \varepsilon_{0}\Vert_{\zeta}^2\right)^{\zeta/2}
\leq \left(C_3'
m^{-2\beta+1}\right)^{\zeta/2},
\end{eqnarray}
with some constants $C_3,C_3'>0$, using the Marczinkiewicz-Zygmund inequality. 
Define $b_m:=\sqrt{C_3'}m^{-\beta+1/2}$. An analogous estimate gives $\gamma^n_{\zeta}(m,X)\leq b_m$ for
all $n\in\mathbb{N}$. Since $\sum_{m=1}^{\infty} b_m<\infty$ by $\beta>3/2$, $\Gamma^n_{\zeta}(X)$ is actually bounded by a constant, uniformly in $n$.
\end{proof}

We also need in the sequel that the law of the driving noise is smooth enough.
This is formulated in terms of the characteristic function $\phi$ of $\varepsilon_0$.

\begin{assumption}\label{smooth}
We require that
\begin{equation}\label{smoothie}
\int_{\mathbb{R}} |\phi(u)|\, du<\infty.
\end{equation}
\end{assumption}

\begin{remark}\label{sm}
{\rm Assumption \ref{smooth} implies the existence of a (continuous and bounded) density $f$ for the
law of $\varepsilon_0$ (with respect to the Lebesgue measure). Indeed, $f$ is the inverse
Fourier transform of $\phi$:
$$
f(x)=\frac{1}{2\pi}\int_{\mathbb{R}} \phi(u)e^{-iux}\, du,\ x\in\mathbb{R}.
$$
Conversely, if the law of $\varepsilon_0$ has a
twice continuously differentiable density $f$ such
that $f'$, $f''$ are integrable over $\mathbb{R}$ then \eqref{smoothie} holds. The latter observation
follows by standard Fourier-analytic arguments.}
\end{remark}

\begin{lemma}\label{tak} Let Assumptions \ref{linear} and \ref{smooth} be in force.
Then the law of $X_0$ (resp. $X^+_{0,m}$) has a density $f_{\infty}$ (resp. $f_m$) with respect to the Lebesgue measure.
Moreover, there is a constant $\tilde{K}>0$ such that
$$
\sup_{m\in\mathbb{N}\cup\{\infty\}}\sup_{x\in\mathbb{R}} f_m(x)\leq \tilde{K}.
$$
\end{lemma}
\begin{proof} Denote by $\phi_m$ the characteristic function of 
$X^+_{0,m}$. Since $|\phi(u)|\leq 1$ for all $u$, we see that
 \begin{equation}\label{ackroyd}
 \vert \phi_m(u)\vert=\left\vert \prod_{j=0}^{m-1} \phi(a_j u)\right\vert\leq |\phi(a_0 u)|,
 \end{equation}
 which implies, by applying an inverse Fourier transform, the existence of $f_m$ and the estimate
 $$
 |f_m(x)|\leq \frac{1}{2\pi} \int_{\mathbb{R}}|\phi(a_0 u)|\, du=:\tilde{K}<\infty,\mbox{ for all }x\in\mathbb{R},\ m\geq 1,
 $$
 by Assumption \ref{smooth}. As $X_{0,m}^+$ tends to $X_0$ in probability when $m\to\infty$,
 $\phi_m(u)$ tends to $\phi_{\infty}(u)$ for all $u$, where $\phi_{\infty}$ is the characteristic function of $X_0$.
 The integrable bound \eqref{ackroyd} is uniform in $m$, so $f_{\infty}$ exists and 
 the dominated convergence theorem implies that $f_m(x)$ tends to $f_{\infty}(x)$, for all $x\in\mathbb{R}$.
 The result follows.
\end{proof}

Let $D\subset\mathbb{R}^N$ be a bounded open set. In the sequel we consider functionals of the form
\begin{equation}\label{szek}
H(\theta,x):=\sum_{j=1}^M g_j(\theta,x) 1_{\{x\in I_j(\theta)\}},\ x\in\mathbb{R},\ \theta\in D,
\end{equation}
where the $g_j$ are bounded and Lipschitz-continuous functions (jointly in the two
variables) and the intervals $I_j(\theta)$ are of the form
$(-\infty,h_j(\theta))$,
$(h_j(\theta),\infty)$ or $(h^1_j(\theta),h^2_j(\theta))$ with
$h_j,h^1_j,h^2_j:D\to \mathbb{R}$ Lipschitz-continuous functions.

\begin{remark}
{\rm The intervals $I_j(\theta)$ can also be
closed or half-closed and the
results below remain valid, this is clear from the proofs. In the one-dimensional case, the signed regressor,
signed error, sign-sign and Kohonen algorithms all have an updating function of the form \eqref{szek}, see \cite{e}, \cite{e2}, 
\cite{kohonen1}, \cite{kohonen}.
For simplicity,
we only treat the one-dimensional setting (i.e. $x\in\mathbb{R}$) in the present paper but we allow $D$ to
be multidimensional.}
\end{remark}

\begin{lemma}\label{lem_CLC} Let Assumptions \ref{linear} and \ref{smooth} be in force. 
Then a random field $H(\theta,X_t)$,
$t\in\mathbb{N}$, $\theta\in D$ as in \eqref{szek} satisfies
the CLC property \eqref{kac}.
\end{lemma}
\begin{proof}
It suffices to consider $H(\theta,X_1)= g(\theta,X_1)1_{\{X_1 \in I(\theta)\}}$ 
with $g$ Lipschitz-continuous, bounded and
$I$ of the form $(-\infty,h(\theta))$, $(h(\theta),\infty)$ or $(h^1(\theta),h^2(\theta))$ with $h,h^1,h^2$ Lipschitz.
We only prove the first case, the other cases being similar. Recall also Remark \ref{talann}.

Denoting by $C_4$ a Lipschitz-constant for $g$ and by $C_5$ an upper bound for $|g|$, we get the estimate
\begin{eqnarray*}
|H(\theta_1,X_1)-H(\theta_2,X_1)| &\leq&\\
|1_{\{X_1< h(\theta_1)\}} g(\theta_1,X_1)
-1_{\{X_1< h(\theta_1)\}} g(\theta_2,X_1) | &+&\\
|1_{\{X_1< h(\theta_1)\}} g(\theta_2,X_1)-
1_{\{X_1< h(\theta_2)\}} g(\theta_2,X_1)| &\leq&\\ C_4 |\theta_1-\theta_2| + C_5
| 1_{\{X_1< h(\theta_1)\}}-1_{\{X_1< h(\theta_2)\}} | &\leq&\\
C_4 |\theta_1-\theta_2| + C_5\left(
1_{\{X_1\in [h(\theta_1),h(\theta_2))\}} + 1_{\{X_1\in [h(\theta_2),h(\theta_1))\}}\right). & &
\end{eqnarray*}
We may and will assume $h(\theta_1)<h(\theta_2)$. It suffices to prove that
$$
P\left(X_1\in [h(\theta_1),h(\theta_2))\vert\mathcal{F}_0\right)\leq C_{6} |\theta_1-\theta_2|
$$
with a suitable $C_6>0$. Noting that the density of $a_0 \varepsilon_1$ is $x\to (1/\vert a_0\vert) f(x/a_0)$, we have
\begin{eqnarray*}
P\left(X_1\in [h(\theta_1),h(\theta_2))\vert\mathcal{F}_0\right) &=&
\int_{h(\theta_1)-\sum_{j=1}^{\infty} a_j\varepsilon_{1-j}}^{h(\theta_2)-\sum_{j=1}^{\infty} a_j\varepsilon_{1-j}}
\frac{1}{\vert a_0\vert}{f}(x/a_0)\, dx \leq\\
\frac{1}{\vert a_0\vert}K_0 | h(\theta_1)-h(\theta_2)| &\leq& \frac{1}{\vert a_0\vert}K_0 C_7 |\theta_1-\theta_2|, 
\end{eqnarray*}
where $K_0$ is an upper bound for $f$ (see Remark \ref{sm}) and
$C_7$ is a Lipschitz constant for $h$. This completes the proof.
\end{proof}

\begin{theorem}\label{exa} Let Assumptions \ref{linear} and \ref{smooth} be in force. Let $H$ be of the
form specified in \eqref{szek}.
Let $\zeta\geq 2r$ and let $\beta$ satisfy $\beta>4r+1/2$.
Then the random field $H(\theta,X_t)$, $t\in\mathbb{N}$, $\theta\in D$ is UCLM-$(r,1)$.
\end{theorem}

\begin{proof} We may and will assume $$
H(\theta,X_t)=g(\theta,X_t)1_{\{X_t\in I(\theta)\}}
$$ 
with some bounded
Lipschitz function $g$ and with some interval $I(\theta)$ of the type as in \eqref{szek}. 
As $H$ is bounded, ${M}_r^n(H)$, $n\in\mathbb{N}$ is trivially
a bounded sequence in $L^1$.

In view of \eqref{betlehem1}, \eqref{betlehem2}, it suffices to establish that $1_{\{X_t\in I(\theta)\}}$,
$t\in\mathbb{N}$ is UCLM-$(2r,1)$ (since $g(\theta,X_t)$, $1_{\{X_t\in I(\theta)\}}$ are bounded and
$g(\theta,X_t)$ is UCLM-$(2r,1)$, by $\zeta\geq 2r$, Lemma \ref{x} and 
Remark \ref{kony1}).
We show this for $I(\theta)=(-\infty,h(\theta))$ with $h$ Lipschitz-continuous
as other types of intervals can be handled similarly.

As $$
\mathrm{Law}(X_{t+n},\, t\in\mathbb{N},\, \varepsilon_{n-j},\, j\in\mathbb{N})=
\mathrm{Law}(X_{t},\, t\in\mathbb{N},\, \varepsilon_{-j},\, j\in\mathbb{N})
$$ 
for all $n\in\mathbb{N}$, we may reduce the proof to estimations for the case $n:=0$.
Let us start with 
\begin{eqnarray*}
\left| 1_{\{ X_{t,m}^{\circ}<h(\theta)\}}-1_{\{ X_t<h(\theta)\}}\right| &=&\\
1_{\{ X_{t,m}^{\circ}<h(\theta),X_t\geq h(\theta)\}}+1_{\{ X_t<h(\theta),X_{t,m}^{\circ}\geq h(\theta)\}} &\leq&\\
1_{\{ X_{t}\in (h(\theta)-\eta_m,h(\theta)+\eta_m)\}}+
1_{\{ |X_{t,m}^{\circ}-X_t|\geq\eta_m\}},
\end{eqnarray*}
for all $\eta_m>0$. We will choose a suitable $\eta_m$ later. 
Using Lemma \ref{tak} and the conditional Markov inequality we obtain
\begin{eqnarray}\nonumber
E\left[| 1_{\{ X_{t,m}^{\circ}<h(\theta)\}}-1_{\{ X_t<h(\theta)\}}|^{2r}\vert\mathcal{F}_0\right] &\leq&\\ \nonumber C_8
P(X_{t}\in (h(\theta)-\eta_m,h(\theta)+\eta_m)\vert\mathcal{F}_0) &+&\\
\nonumber C_8 P(|X_{t,m}^{\circ}-X_t|\geq\eta_m\vert\mathcal{F}_0) &\leq&\\
2C_8 \tilde{K}\eta_m + C_8 E\left[|X_t-X_{t,m}^{\circ}|\big\vert\mathcal{F}_0\right]/\eta_m,& &\label{marco}
\end{eqnarray}
with some constant $C_8$, noting that powers of indicators are themselves indicators and that the conditional density of
$X_t$ with respect to $\mathcal{F}_0$ is $x\to f_t(x-\sum_{j=t}^{\infty}a_j \varepsilon_{t-j})$ and the latter is 
$\leq \tilde{K}$ by Lemma \ref{tak}. Using (\ref{ghjk}), the second term in
\eqref{marco} is bounded by 
$C_8 \sqrt{C_3'} m^{-\beta+1/2}/\eta_m$ hence
it is reasonable to choose $\eta_m:=1/m^{(\beta-1/2)/2}$, which leads to
$$
E^{1/2r}\left[\vert 1_{\{ X_{t,m}^{\circ}<h(\theta)\}}-1_{\{ X_t<h(\theta)\}}\vert^{2r}\big\vert\mathcal{F}_0\right]\leq C_9/m^{(\beta-1/2)/(4r)},
$$
with some $C_9>0$. Notice that $X_{t,m}^{\circ}$ is $\mathcal{F}^+_{t-m}\vee\mathcal{F}_0$-measurable.
Lemma \ref{mall} implies that $\gamma_{2r}(X,m)\leq 2C_9/m^{(\beta-1/2)/(4r)}$. 
As $(\beta-1/2)/(4r)>1$ by our hypotheses, we obtain the UCLM-$(2r,1)$ property
for $1_{\{X_t\in I(\theta)\}}$.
\end{proof}

\begin{remark}\label{strife}{\rm When $\varepsilon_0$ has
moments of all orders then one can reduce the lower bound $4r+1/2$ for $\beta$ in Theorem \ref{exa} to $r+1/2$.
Indeed, in this case $g(\theta,X_t)$ is UCLM-$(q,1)$ for arbitrarily large $q$ by Lemma \ref{x} and Remark \ref{kony1} 
so it suffices to show the UCLM-$(r',1)$ property for $1_{\{ X_t\in I(\theta)\}}$ 
for some $r'>r$ that can be arbitrarily close to $r$ (and not for $r'=2r$ as in Theorem \ref{exa}). The estimate of the above proof can
be improved to
\begin{eqnarray*}
E\left[| 1_{\{ X_{t,m}^{\circ}<h(\theta)\}}-1_{\{ X_t<h(\theta)\}}|^{r'}\vert\mathcal{F}_0\right] &\leq&\\
2C_8 \tilde{K}\eta_m + C_8 E\left[|X_t-X_{t,m}^{\circ}|^q\big\vert\mathcal{F}_0\right]/\eta_m^q,& &
\end{eqnarray*}
for arbitrarily large $q$. Choosing $\eta_m:=1/m^{[q(\beta-1/2)]/(q+1)}$,
we arrive at
$$
E^{1/r'}\left[\vert 1_{\{ X_{t,m}^{\circ}<h(\theta)\}}-1_{\{ X_t<h(\theta)\}}\vert^{r'}\big\vert\mathcal{F}_0\right]
\leq C_9/m^{[q(\beta-1/2)]/[(q+1)r']}.
$$
Let $\beta>r+1/2$. If $r'>r$ 
is chosen close enough to $r$ and $q$ is chosen large enough then $[q(\beta-1/2)]/[(q+1)r']>1$ which shows the
UCLM-$(r,1)$ property for $H(\theta,X_t)$.}
\end{remark}

\begin{lemma}\label{mindet} Let Assumptions \ref{linear} and \ref{smooth} be in force, let $\beta>3/2$. Then, for all 
$n\in\mathbb{N}$,
$$
E\left[\sum_{k=n}^{\infty} \sup_{\vartheta\in D}
\left\vert E\left[H(\vartheta,X_{k+1})\vert\mathcal{F}_n\right]-G(\vartheta)\right\vert \right] \leq C_{10}
$$
with some fixed $C_{10}<\infty$. That is, Assumption \ref{infsum} holds.
\end{lemma}
\begin{proof}
We need to estimate
$$
\sum_{k=n}^{\infty} \left[ \left\vert E\left[g(\vartheta,X_{k+1})1_{\{X_{k+1}<h(\vartheta)\}}\vert\mathcal{F}_n\right]-E[g(\theta, X_{k+1})1_{\{X_{k+1}<h(\theta)\}}]\vert_{\theta=\vartheta} 
\right\vert\right],
$$
where $h:D\to\mathbb{R}$ is Lipschitz-continuous and $g$ is a bounded, Lipschitz-continuous function with a bound $C_{11}$ 
for $\vert g\vert$ and
with Lipschitz constant $C_{12}$. It suffices to prove
$$
E\left[\sup_{\vartheta\in D}\sum_{k=1}^{\infty} \left\vert E\left[g(\vartheta,X_{0})1_{\{X_{0}<h(\vartheta)\}}\vert\mathcal{F}_{-k}\right]-E[g(\theta,X_0)1_{\{X_{0}<h(\theta)\}}]\vert_{\theta=\vartheta} 
\right\vert\right] <\infty,
$$
since the law of $(X_0,\varepsilon_{-k},\varepsilon_{-k-1},\ldots)$ equals that of
$(X_{n+k},\varepsilon_n,\varepsilon_{n-1},\ldots)$, for all $k\geq 1$, $n\in\mathbb{Z}$.
We can estimate a given term in the above series as follows:
\begin{eqnarray}\nonumber
\left\vert E\left[g(\vartheta, X_{0})1_{\{X_{0}<h(\vartheta)\}}\vert\mathcal{F}_{-k}\right]-
E[g(\theta,X_0)1_{\{X_{0}<h(\theta)\}}]\vert_{\theta=\vartheta}\right\vert   &\leq&\\
\left|E\left[g(\vartheta,X_{0})1_{\{X_{0}<h(\vartheta)\}}\vert\mathcal{F}_{-k}\right]-
E[g(\vartheta,X_{0,k}^+)1_{\{ X_{0,k}^+<h(\vartheta)\}}\vert\mathcal{F}_{-k}]\right|  &+&\nonumber\\
\left|E[g(\theta,X_{0,k}^+)1_{\{X_{0,k}^+<h(\theta)\}}]\vert_{\theta=\vartheta}-
E[g(\theta,X_0)1_{\{X_{0}<h(\theta)\}}]\vert_{\theta=\vartheta}\right|  &\leq& \nonumber\\
\nonumber C_{12} E\left[|X_0-X_{0,k}^+|\Big\vert\mathcal{F}_{-k}\right] +C_{11}E\left[
|1_{\{ X_{0,k}^+<h(\vartheta)\}}-1_{\{ X_0<h(\vartheta)\}}\vert\Big\vert\mathcal{F}_{-k}\right] &+&\\
\label{buksza} C_{12}|X_0-X_{0,k}^+| +C_{11}
E\left[ |1_{\{ X_{0,k}^+<h(\theta)\}}-1_{\{ X_0<h(\theta)\}}|\right]\vert_{\theta=\vartheta}  & &
\end{eqnarray}
noting that $E[g(\theta,X_{0,k}^+)1_{\{X_{0,k}^+<h(\theta)\}}]\vert_{\theta=\vartheta}=
E[g(\vartheta,X_{0,k}^+)1_{\{X_{0,k}^+<h(\vartheta)\}}\vert\mathcal{F}_{-k}]$.
The first and third terms on the right-hand side of \eqref{buksza} are equal and they are $\leq C_{13} k^{-\beta+1/2}$
with some $C_{13}>0$, by the proof of Lemma \ref{x}, hence their sum (when $k$ goes from $1$ to infinity) is finite. 
The expression in the second term of \eqref{buksza} can be estimated as 
\begin{eqnarray}\nonumber
E\left[|1_{\{ X_{0,k}^+<h(\vartheta)\}}-1_{\{ X_0<h(\vartheta)\}}\vert\Big\vert\mathcal{F}_{-k}\right] &\leq&\\
\nonumber P\left(X_0\in \left(h(\vartheta)-\sum_{j=k}^{\infty}a_j\varepsilon_{-j},h(\vartheta)+\sum_{j=k}^{\infty}a_j\varepsilon_{-j}\right)\vert\mathcal{F}_{-k}\right)
&\leq&\\
2\tilde{K}\left\vert\sum_{j=k}^{\infty}a_j\varepsilon_{-j}\right\vert,\label{hora}
\end{eqnarray}
noting that 
the conditional density of $X_0$ with respect to $\mathcal{F}_{-k}$ is
$x\to f_k\left(x-\sum_{j=k}^{\infty}a_j\varepsilon_{-j}\right)$ and this is bounded
by $\tilde{K}$, using Lemma \ref{tak}. Since \eqref{hora} is independent of $\vartheta$, a 
similar estimate guarantees that
$$
E\left[ |1_{\{ X_{0,k}^+<h(\theta)\}}-1_{\{ X_0<h(\theta)\}}|\right]\vert_{\theta=\vartheta}\leq 2\tilde{K}\left\vert\sum_{j=k}^{\infty}a_j\varepsilon_{-j}\right\vert.
$$
Note that the upper estimates obtained so far do not depend on $\vartheta$.
It follows that, even taking supremum in $\vartheta\in D$, the expectations of the second and fourth terms
on the right-hand side of \eqref{buksza} are both $\leq C_{14} k^{(-\beta+1/2)}$
with some $C_{14}>0$. As $\beta>3/2$, the infinite sum of these terms is finite, too, finishing the
proof of the present lemma.
\end{proof}

\begin{assumption}\label{call}
Let $f$ satisfy
\begin{equation}\label{krudy}
\vert f(x)\vert\leq \widehat{C} e^{-\widehat{\delta} |x|}\mbox{ for all }x\in\mathbb{R},
\end{equation}
with some $\widehat{C},\widehat{\delta}>0$
\end{assumption}

\begin{assumption}\label{vall}
Let
\begin{equation}\label{ko}
\int_{\mathbb{R}} u^2\vert\phi_u\vert \, du<\infty
\end{equation}
hold.
\end{assumption}

\begin{remark}\label{prodigue} {\rm Clearly, Assumption \ref{call} implies that $\varepsilon_0$ has finite moments of all orders.
Note also that Assumption \ref{vall} implies Assumption \ref{smooth}.}
\end{remark}

\begin{remark}{\rm If $f$ is four times continuously
differentiable such that $f'$, $f''$, $f'''$, $f''''$ are integrable then \eqref{ko} holds, compare to Remark \ref{sm}
above.}
\end{remark}

\begin{lemma} Let Assumptions \ref{linear}, \ref{call} and \ref{vall} hold.
Let the functions $g_i$, $h^1_i$, $h^2_i$, $h_i$ of \eqref{szek} be twice continuously 
differentiable with bounded first and second
derivatives.
Then $G(\theta):=EH(\theta,X_0)$
is bounded and twice continuously differentiable with bounded first and second
derivatives, i.e. Assumption \ref{boun} holds.
\end{lemma}
\begin{proof} We may and will assume that
$$
H(\theta,x)=1_{\{x<h(\theta)\}}g(\theta,x)
$$
with $h$ Lipschitz-continuous, $g$ bounded and Lipschitz-continuous. $G$ is bounded since $g$ is.
We proceed to establish its differentiability and the boundedness of its derivatives.

Recall that
$$
\phi_{\infty}(u)=\prod_{j=0}^{\infty} \phi(a_j u),\ u\in\mathbb{R},
$$
where $\phi_{\infty}$ is the characteristic function of $X_0$ and the product converges pointwise.
Since $|\phi(u)|\leq 1$ for all $u$,
\eqref{ko} implies that
$$
\int_{\mathbb{R}} u^2\vert\phi_{\infty}(u)\vert\, du<\infty.
$$
Clearly, this implies $\int_{\mathbb{R}} |u\phi_{\infty}(u)|\, du<\infty$ and
$\int_{\mathbb{R}} |\phi_{\infty}(u)|\, du<\infty$ as well (since $\phi_{\infty}$ is bounded,
being a Fourier transform).
Now one can directly show, using the inverse Fourier transform, that $f_{\infty}$, the density
of the law of $X_0$, is twice continuously differentiable.

Inequality \eqref{krudy} implies that $\phi$ has a complex analytic extension in a strip around $\mathbb{R}$.
Since the sequence $a_j$, $j\in\mathbb{N}$ is bounded, there is even a strip such that $u\to \phi(a_j u)$ is analytic in it,
for all $j\in\mathbb{N}$,
thus $\phi_{\infty}$ is also analytic there. Then so are $-iu\phi_{\infty}(u)$ and $-u^2\phi_{\infty}(u)$.
These being integrable, we get that their inverse Fourier transforms,
$f_{\infty}'$ and $f_{\infty}''$, satisfy
\begin{equation}\label{uccu}
\vert f_{\infty}'(x)\vert+\vert f_{\infty}''(x)\vert\leq \tilde{C}e^{-\tilde{\delta}|x|}\mbox{ for all }x\in\mathbb{R},
\end{equation}
with some $\tilde{C},\tilde{\delta}>0$, see e.g. Theorem 11.9.3 of \cite{kawata}. In particular, $f_{\infty}',f_{\infty}''$ are integrable.

For notational simplicity we consider only the case $N=1$, i.e. $D\subset\mathbb{R}$.
Using the change of variable $y=x-h(\theta)$, we see that
$$
EH(\theta,X_0)=\int_{\mathbb{R}} g(\theta,x)1_{\{x<h(\theta)\}}f_{\infty}(x)\, dx=\int_{-\infty}^0
g(\theta,y+h(\theta))f_{\infty}(y+h(\theta))\, dy.
$$
We calculate $\partial_{\theta}g(\theta,y+h(\theta))f_{\infty}(y+h(\theta))$:
$$
[\partial_1 g(\theta,y+h(\theta))+ \partial_2 g(\theta,y+h(\theta))h'(\theta)]f_{\infty}(y+h(\theta)) +
g(\theta,y+h(\theta))f_{\infty}'(y+h(\theta))h'(\theta),
$$
where $\partial_1$ (resp. $\partial_2$) denote differentiation with respect to the first (resp. second)
variable.
As $f_{\infty}$ (resp. $f_{\infty}'$) satisfy \eqref{krudy} (resp. \eqref{uccu}) and $g,\partial_1 g,\partial_2 g, h'$ are bounded, the dominated convergence theorem implies that
\begin{eqnarray*}
\partial_{\theta} EH(\theta,X_0) &=&\\
\int_{-\infty}^0[\partial_1 g(\theta,y+h(\theta))+\partial_2
g(\theta,y+h(\theta))h'(\theta)]f_{\infty}(y+h(\theta))\, dy &+&\\
\int_{-\infty}^0 g(\theta,y+h(\theta))f_{\infty}'(y+h(\theta))h'(\theta)\, dy &=&\\
\int_{\mathbb{R}}1_{\{x<h(\theta)\}}[\partial_1 g(\theta,x)+\partial_2
g(\theta,x)h'(\theta)]f_{\infty}(x)\, dx &+&\\
\int_{\mathbb{R}} 1_{\{x<h(\theta)\}}g(\theta,x)f_{\infty}'(x)h'(\theta)\, dx,& &
\end{eqnarray*}
where both integrals are clearly bounded in $\theta$. Similar calculations involving the second derivatives of $g,h,f_{\infty}$
show that $\partial_{\theta}^2 EH(\theta,X_0)$ exists and it is bounded in $\theta$.
\end{proof}

The following corollary summarizes our findings in the present section.

\begin{corollary}\label{mao} Let $H$ be of the form \eqref{szek} such that $g_j,h_j,h^1_j,h^2_j$ are
twice continuously differentiable with bounded first and second derivatives.
Let Assumptions \ref{linear}, \ref{call} and \ref{vall} hold and assume $\beta>5/2$.
Then Theorem \ref{main} applies to the random field $H(\theta,X_t)$, $t\in\mathbb{N}$, $\theta\in D$, provided
that Assumption \ref{stab} holds.
\end{corollary}
\begin{proof}
Recalling \ref{prodigue} and \ref{strife}, this corollary
follows from the results of the present section.
\end{proof}

Assumptions \ref{call} and \ref{vall} apply, in particular, 
when $\varepsilon_0$ is Gaussian. There does not seem 
to be a general condition guaranteeing the validity of Assumption \ref{stab}: this needs
checking in every concrete application of Theorem \ref{main}.

\subsection{Markov chains in a random environment}\label{mark}

If $\beta\leq 5/2$ in the setting of Subsection \ref{lini} above then Corollary \ref{mao} cannot
be established with our methods. Hence
$X_t$ cannot be a ``long memory processes'' in the sense of \cite{giraitis}. In this subsection we show that
it is nonetheless possible to apply Theorem \ref{main} to important classes of random fields that are driven by a 
long memory process, see
Example \ref{rough} below.

Let $\varepsilon^1_t$, $t\in\mathbb{Z}$, $\varepsilon^2_t$, $t\in\mathbb{Z}$ be i.i.d. real-valued
sequences, independent of each other.

\begin{assumption}\label{contract}  
We denote $\chi_t:=(\varepsilon^1_{j+t})_{j\in\mathbb{Z}}$, for each $t\in\mathbb{Z}$.
Let $F:\mathbb{R}^{\mathbb{Z}}\times\mathbb{R}\times\mathbb{R}^m\to\mathbb{R}^m$ be a measurable function
such that, for all $w\in\mathbb{R}^{\mathbb{Z}}$, $s\in\mathbb{R}$,
$$
\vert F(w,s,z_1)-F(w,s,z_2)\vert\leq \rho |z_1-z_2|,
$$
for all $z_1,z_2\in\mathbb{R}^m$ with some $0<\rho<1$. Furthermore, there is $x\in\mathbb{R}^m$ such that for all 
$w\in\mathbb{R}^{\mathbb{Z}}$ and for all $s\in\mathbb{R}$,
$$
| F(w,s,x)-x |\leq C(1+\vert s\vert)
$$
and $E\vert \varepsilon_0^2\vert^r<\infty$ for some $r>2$.
\end{assumption}

Fix $x\in\mathbb{R}$ as in Assumption \ref{contract} and define, for all 
$t\in\mathbb{Z}$, $\tilde{X}^t_0:=x$, and for $j\geq 0$,
$$
\tilde{X}^t_{j+1}:=F(\chi_{t},\varepsilon_{t}^2,\cdot)\circ F(\chi_{t-1},\varepsilon_{t-1}^2,\cdot)\circ\cdots
F(\chi_{t-j},\varepsilon_{t-j}^2,\cdot)(x).
$$

Standard arguments (such as Proposition 5.1 of \cite{diaconis}) show that $\tilde{X}_j^t$
converges almost surely as $j\to\infty$. Define
$$
X_t:=\lim_{j\to\infty} \tilde{X}^t_j. 
$$
Then $X_t$, $t\in\mathbb{Z}$ is clearly a stationary process, satisfying
$$
X_{t+1}=F(\chi_{t+1},\varepsilon_{t+1}^2,X_t),\ t\in\mathbb{Z}.
$$

When freezing the values of $\varepsilon_t^1$, $t\in\mathbb{Z}$, the $X_t$ defined above
is an (inhomogeneous) Markov chain driven by the noise sequence $\varepsilon_t^2$,
$t\in\mathbb{Z}$. Hence $X_t$ is a Markov
chain in a random environment (the latter is driven by $\varepsilon_t^1$, $t\in\mathbb{Z}$).

\begin{example}\label{rough} {\rm Let $\varepsilon_i^2$, 
$i\in\mathbb{Z}$ be i.i.d. with $E|\varepsilon_0^2|^r<\infty$
for some $r>2$. Let $E\varepsilon_0^1=0$, $E(\varepsilon_0^1)^2<\infty$.
Let $Y_t:=\sum_{j=0}^{\infty} a_j \varepsilon_{t-j}^1$, $t\in\mathbb{Z}$ for some $a_j$, $j\in\mathbb{N}$
with $\sum_{j=0}^{\infty} a_j^2<\infty$. The series converges almost surely. Let $h_1,h_2:\mathbb{R}\to\mathbb{R}$
be bounded measurable and fix $-1<\kappa,\rho<1$. The construction sketched above provides the existence of a process $X_t$
satisfying
$$
X_{t+1}=\kappa X_t +\rho e^{h_1(Y_{t+1})}h_2(\varepsilon_{t+1}^1)+\sqrt{1-\rho^2}e^{h_1(Y_{t+1})}\varepsilon_{t+1}^2.
$$

This is an instance of stochastic volatility models where $h_1(Y)$ corresponds to the log-volatility 
of an asset and $X$ is the increment of the log-price of the same 
asset. Note that $Y$ may have a slow
autocorrelation decay (e.g. $a_j\sim j^{-\beta}$ with any $\beta>1/2$ is possible).
This model resembles the ``fractional stochastic volatility model'' of \cite{comte,rough}.
Choose $x:=0$
and 
$$
F(w,s,z):=\kappa z+\rho e^{h_1(\sum_{j=0}^{\infty}a_j w_{t+1-j})}h_2(w_{t+1})
+\sqrt{1-\rho^2}e^{h_1(\sum_{j=0}^{\infty}a_j w_{t+1-j})}s.
$$
As easily seen, Assumption \ref{contract}
holds for this model and thus Theorem \ref{louvre} below applies.

The functions $h_1,h_2$ serve as 
truncations only, in order to satisfy Assumption \ref{contract}. One could probably
relax Assumption \ref{contract} to accomodate the case $h_1(x)=h_2(x)=x$ as well. We refrain from
the related complications in the present paper.} 
\end{example}

The result below permits to estimate the tracking error 
for another large class of
non-Markovian processes. For simplicity, we consider only smooth functions $H$ here.

\begin{theorem}\label{louvre} Let 
$D\subset\mathbb{R}^N$ be bounded and open.
Let Assumption \ref{contract} hold. Let $H:D\times\mathbb{R}^m\to\mathbb{R}^N$ be
bounded, twice continuously differentiable, with bounded first and second derivatives.
Then the conclusion of Theorem \ref{main} is true for $H(\theta,X_t)$, $t\in\mathbb{N}$, $\theta\in D$, provided
that Assumption \ref{stab} holds. 
\end{theorem}

The proof is given in Section \ref{appi}.
Most results in the literature are about homogeneous (controlled) Markov chains 
hence they do not apply to the present, inhomogeneous case and we exploit the $L$-mixing property in an essential way
in our arguments. See, however, also Subsection 5.3 of \cite{lp} for alternative conditions in the inhomogeneous Markovian case.

\section{Numerical implementation}\label{numera}

Numerical results are presented here verifying the convergence properties of stochastic approximation procedures
with a fixed gain in the case of discontinuous $H$, for Markovian and non-Markovian models. The purpose here is 
illustrative. 

\subsection{Quantile estimation for AR($1$) processes}

We first consider a Markovian example in the simplest possible case where
$H(\theta,\cdot)$ is an indicator function. Let $X_t$, $t \in \mathbb{Z}$ be an AR(1) process defined by
\begin{align*}
X_{t+1}=\alpha X_{t}+\varepsilon_{t+1}
\end{align*}
where $\alpha$ is a constant satisfying $|\alpha| <1$ and $\varepsilon_{t}$, $t \in \mathbb{Z}$ are i.i.d standard normal variates. As a consequence of the above equation, one observes that
\begin{align*}
X_{t}=\sum_{j=0}^\infty \alpha^{j} \varepsilon_{t-j}
\end{align*}
for every $t \in \mathbb{N}$. Moreover, $X_t$ has stationary distribution which is $\nu:=N(0,(1-\alpha^2)^{-1})$ and the pair $(X_t, X_{t+1})$ has bivariate normal distribution with correlation $\alpha$. We are interested in finding the quantile of the stationary distribution $\nu$ using the stochastic approximation method \eqref{elso} with fixed gain.

The algorithm for the fixed gain $\lambda>0$ is given by the following equation,
\begin{align} \label{num:theta}
\theta_{t+1}=\theta_{t}+\lambda H(\theta_t, X_{t+1}),
\end{align}
for every $t \in \mathbb{N}$. For the purpose of the $q$-th quantile estimation of the stationary distribution $\nu$, one takes
\begin{align} \label{num:H}
H(\theta,x)= q - 1_{\{x \leq \theta\}}.
\end{align}
With this choice of $H$, the solution of (\ref{g}) is the quantile in question. The function $H$ is just the gradient of the so-called ``pinball'' loss function introduced in Section 3 of \cite{rb} for quantile estimation. 
The true value of the $q$-th quantile of $\nu$ is $\Phi(q)/\sqrt{1-\alpha^2}$, where $\Phi$ is the cumulative distribution function of the standard normal variate. For our numerical experiments, we take $\alpha=0.5$ and $q=0.975$ and hence the true value of the $q$-th quantile is $\theta^*\approx 2.26$.

Figure \ref{fig:fg1} illustrates that the rate of convergence of the fixed gain algorithm is consistent with our theoretical findings in the paradigm of the quantile estimation of the stationary distribution of an AR$(1)$ process. As noted above, the true value of the quantile in this particular example is $2.26$ which is then compared with the estimate obtained by using the fixed gain approximation algorithm. The Monte Carlo estimate 
is based on $12000$ samples and the number of iterations is taken to be $I = 10^6$ with initial value $\theta_0=2.0$.

\begin{figure}[H]
\includegraphics[scale=0.6]{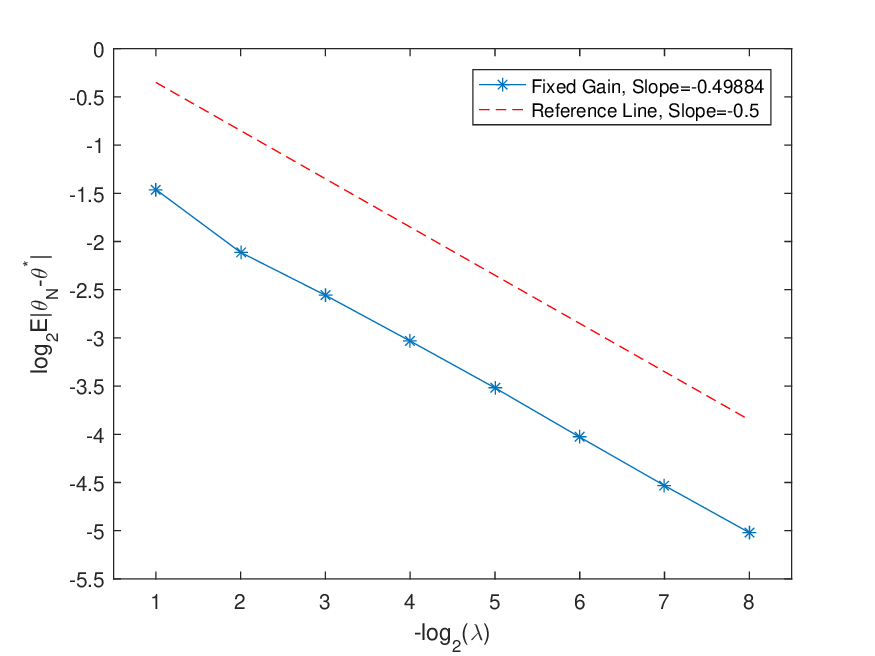}
\caption{Rate of Convergence of Fixed Gain Algorithm for AR($1$) Process}\label{fig:fg1}
\end{figure}

\subsection{Quantile estimation for MA($\infty$) processes}

Let us now consider the case when $X_t$, $t \in \mathbb{N}$ is an MA($\infty$) process which is  
non-Markovian. It is given by
\begin{equation}\label{filu}
X_t=\sum_{j=0}^\infty \frac{1}{(j+1)^\beta} \varepsilon_{t-j}
\end{equation}
where $\beta>1/2$ and $\varepsilon_t$, $t \in \mathbb{Z}$ are i.i.d sequence of standard normal variates. One can notice that the stationary distribution of MA($\infty$) process is given by
\begin{align*}
X_t \sim N\Big(0, \sum_{j=0}^\infty \frac{1}{(j+1)^{2\beta}}\Big)
\end{align*}
for any $t \in \mathbb{N}$. As before, we are interested in the estimation of the quantile of the stationary distribution. In our numerical calculations, $\beta=3$ and the exact variance is $\pi^6/945$.  For generating the path of the MA($\infty$) process, we write $X_t$ as
\begin{align*}
X_t=\sum_{j=0}^{t} \frac{1}{(j+1)^\beta} \varepsilon_{t-j} + \sum_{j=0}^{\infty} \frac{1}{(t+2+j)^\beta} \varepsilon_{-j-1}
\end{align*}
and notice that
\begin{align*}
Y_t:=\sum_{j=0}^{\infty} \frac{1}{(t+2+j)^\beta} \varepsilon_{-j-1} \sim N\Big(0, \sum_{j=0}^{\infty} \frac{1}{(t+2+j)^{2\beta}}\Big)
\end{align*}
for any $t \in \mathbb{N}$. Also, a reasonable approximation of the variance of $Y_t$ can be
\begin{align*}
\mathrm{var}(Y_t) \approx \sum_{j=0}^{12} \frac{1}{(t+2+j)^{2\beta}}
\end{align*}
which is within an interval of length $10^{-7}$ around the true value.  With this set-up, the stochastic approximation method \eqref{num:theta} with updating function \eqref{num:H} is implemented for the quantile estimation of the  stationary distribution of MA($\infty$) process with $\theta_0=2.0$ and $\theta^*=1.976950$ ($0.975$-th quantile). 
Figure \ref{fig:fg2} indicates that the rate of convergence of the fixed gain algorithm is $0.5$, which is consistent with the theoretical findings. The Monte Carlo estimate  
is based on $10^5$ samples.  
Figure \ref{fig:fg2} is based on $I = 10^5$ iterations.

\begin{figure}[H]
\includegraphics[scale=0.6]{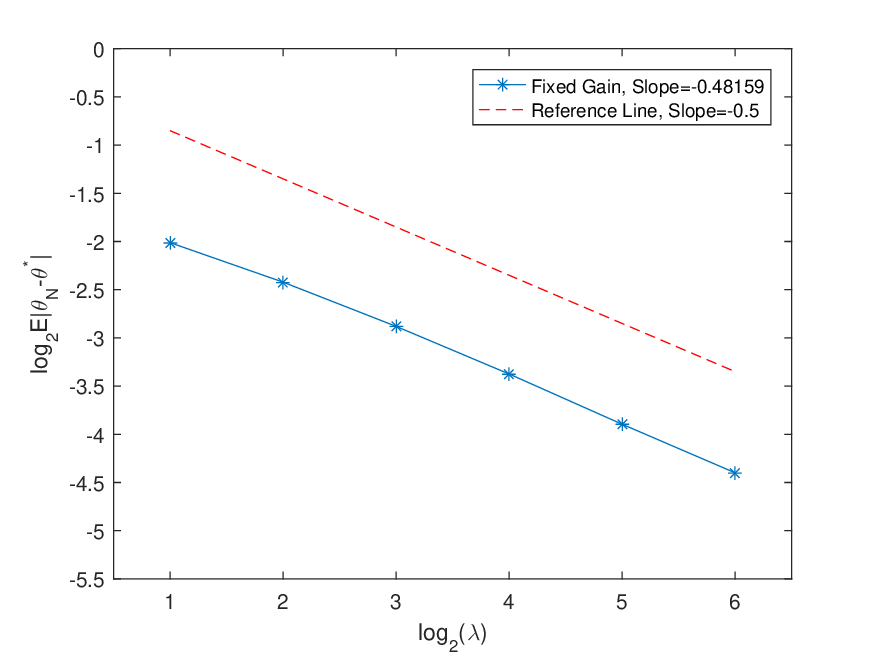}
\caption{Rate of Convergence of Fixed Gain Algorithm for MA($\infty$) Process}\label{fig:fg2}
\end{figure}

\subsection{Kohonen algorithm}

In this section, we demonstrate the rate of convergence of the Kohonen algorithm for 
optimally quantizing a one-dimensional random variable $X$. 
We refer to \cite{kohonen, fort} for discussions. 
We fix the number of cells $N\geq 1$ in advance. Let 
$\theta:=(\theta^1, \ldots, \theta^N)\in\mathbb{R}^N$  
and define Voronoi cells as 
\begin{align*}
\mathcal{V}^i(\theta):=\{ x \in \mathbb{R}: |x-\theta^i|=\min_{j\in\{1,\ldots,N\}}|x-\theta^j|\}
\end{align*}
for $i=1,\ldots,N$. Values of $X$ in a cell $i$ will be
quantized to $\theta^i$. The zero-neighbourhood fixed gain Kohonen algorithm is
aimed at minimizing, in $\theta$, the quantity
$$
\sum_{i+1}^N E\left[\vert X-\theta^i\vert^2 1_{\mathcal{V}^i(\theta)}(X)\right].
$$
Differentiating (formally) this formula suggests the recursive procedure
\begin{align}
\theta_{t+1}^i=\theta_t^i+\lambda \mathbf{1}_{\mathcal{V}^i(\theta)}(Y_t)(Y_t-\theta^i_t) \label{num:koh1}
\end{align}
for every $i=1,\ldots,N$ where $t \in \mathbb{N}$ and the process $Y$ has a stationary distribution 
equal to the law of $X$. The algorithm approximates the $\mathbb{R}$-valued random variable $X$ by $\theta^i$ if 
its values lie in the cell $\mathcal{V}^i(\theta)$, for every $i=1,\ldots,N$. 

In 
Figure  \ref{fig:fig3}, we demonstrate the rate of convergence of the zero-neighbourhood Kohonen algorithm 
with zero-neighbours when the signal $Y_t$s are i.i.d. observations from uniform distribution on $[0,1]$,
which is a well-understood case, see e.g. \cite{kohonen}. 
We take $N=2$, $\theta:=(\theta_1,\theta_2)$, $\mathcal{V}^1(\theta)=(0,(\theta^1+\theta^2)/2]$ and 
$\mathcal{V}^2(\theta)=[(\theta_1+\theta_2)/2,1)$. Hence, the optimal value of $\theta$ is 
$\theta^{1*}=1/4$ and $\theta^{2*}=3/4$. The number of iterations is $10^8$ and the number of sample paths is $10^3$. Furthermore, the initial values of $\theta$s are $\theta^1_0=0.01$ and $\theta^2_0=0.02$. As illustrated, the rate of convergence is close to $0.5$ which is consistent with the theoretical findings. 

\begin{figure}[H]
\includegraphics[scale=0.5]{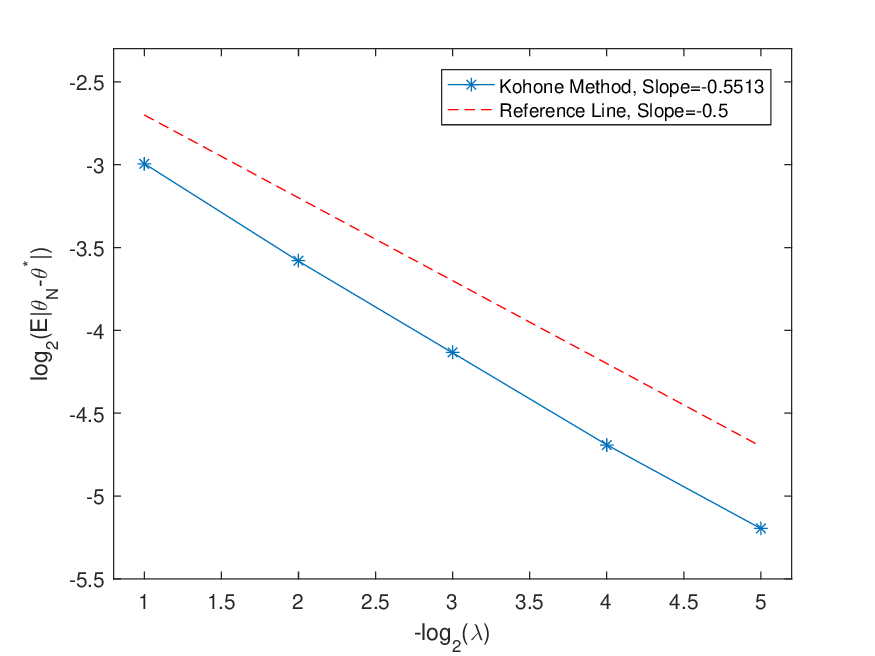}
\caption{Rate of Convergence of Kohonen Algorithm for i.i.d. $U(0,1)$.}\label{fig:fig3}
\end{figure}

Now, to have a non-Markovian example, consider a moving average process with lag $10$, i.e. 
$$
X_t = \sum_{j=0}^{10}\frac{1}{(j+1)^\beta} \varepsilon_{t-j},\ t\in\mathbb{N},
$$ 
where $\varepsilon_t$, $t\in\mathbb{Z}$ are independent standard Gaussian random variables, denote it by MA($10$). Clearly, 
$$
X_t \sim \mathcal{N}\Big(0, \sum_{j=0}^{10}\frac{1}{(j+1)^{2\beta}}\Big)
$$ 
for any $t \geq 0$. Take $\beta:=3$ and notice that MA($10$) is a good approximation of MA($\infty$) process 
\eqref{filu} because the contributions from other terms are negligible due to low variance. We take $N=2$ 
and implement the Kohonen algorithm \eqref{num:koh1} to sample two elements $\theta:=(\theta^1, \theta^2)$ 
from the stationary distribution of the process $Y$ defined by $Y_t:=\tan^{-1}(X_t)$ for any $t \geq 0$. 
As the support of the stationary distribution of the process $Y$ is $(-\pi/2, \pi/2)$, the Voronoi cells are $\mathcal{V}^1(\theta):=(-\pi/2, (\theta^1+\theta^2)/2]$ and $\mathcal{V}^2(\theta):=[(\theta^1+\theta_2)/2, \pi/2)$. The true values  $\theta^*:=(\theta^{1*},\theta^{2*})$ are the solution of the following system of two non-linear equations: 
\begin{align*}
\theta^{1*} \Phi\Big(\frac{1}{\sigma}\tan\big(\frac{\theta^{1*}+\theta^{2*}}{2}\big)\Big)=E\Big(\tan^{-1}(\sigma Z)
\mathbf{1}_{(-\infty, \frac{1}{\sigma}\tan(\frac{\theta^{1*}+\theta^{2*}}{2})]}(Z)\Big)
\\
\theta^{2*} \Big[1-\Phi\Big(\frac{1}{\sigma}\tan\big(\frac{\theta^{1*}+\theta^{2*}}{2}\big)\Big)\Big]=E\Big(\tan^{-1}(\sigma Z)\mathbf{1}_{[\frac{1}{\sigma}\tan(\frac{\theta^{1*}+\theta^{2*}}{2}), \infty )}(Z)\Big)
\end{align*}
where $\sigma^2:=\mathrm{var}(X_t)$,  $Z$ denotes the standard normal variate and $\Phi$ its distribution function. 

Figure \ref{fig:fig4} is based on  $10^8$ iterations and $3000$ paths (for Monte Carlo simulations). 
The initial values are $\theta_0^1=-\pi/4$ and $\theta_0^2=\pi/4$. Since $\theta^*$ is not known the output of
the Kohonen algorithm \eqref{num:koh1} with $\lambda=2^{-9}$ is taken as $\theta^*$. 
Again, our numerical experiments are consistent with the theoretical rate $\lambda^{1/2}$ found in Theorem \ref{main} above.

\begin{figure}[H]
\includegraphics[scale=0.5]{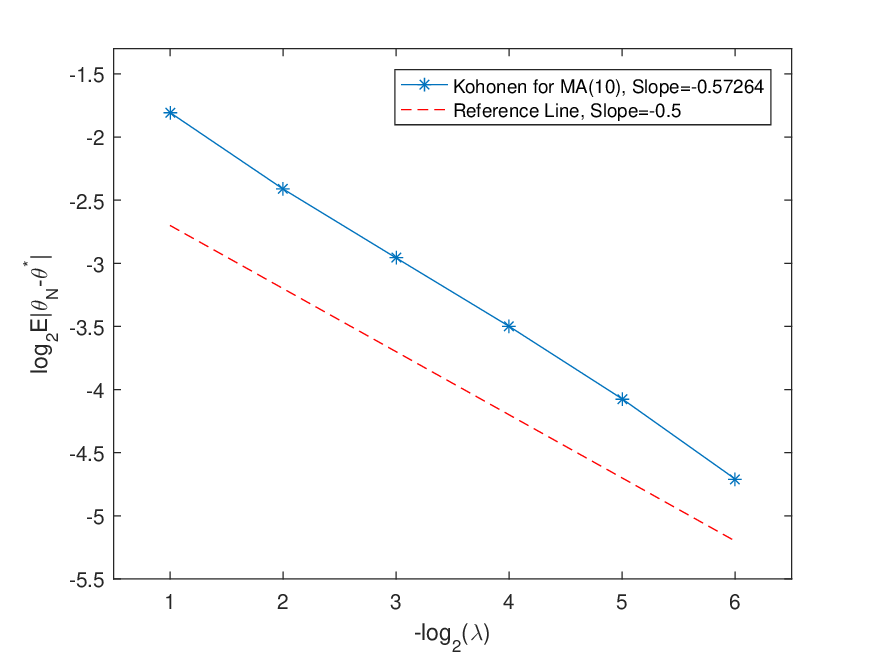}
\caption{Rate of Convergence of Kohonen Algorithm for MA($10$).}\label{fig:fig4}
\end{figure}

\section{Appendix}\label{appi}

Here we gather the proofs for Sections \ref{ketto} and \ref{harom} as well as for Theorem \ref{louvre}.
First we present a slight extension of Lemma 2.1 of \cite{laci1} which is used multiple times.

\begin{lemma}\label{mall} Let 
$\mathcal{G},\mathcal{H}\subset\mathcal{F}$
be sigma-algebras. Let $X,Y$ be random variables in $L^p$ such that $Y$ is measurable with
	respect to $\mathcal{H}\vee\mathcal{G}$.
	Then for any $p\ge 1$,
	$$
	E^{1/p}\left[\vert X-E[X\vert\mathcal{H}\vee\mathcal{G}]\vert^p\big\vert \mathcal{G}\right]
	\leq 2E^{1/p}\left[\vert X-Y\vert^p\big\vert \mathcal{G}\right].
	$$
	If $Y$ is $\mathcal{H}$-measurable then
	\begin{equation}\label{dirk}
	\Vert X-E[X\vert\mathcal{H}]\Vert_p\leq 2 \Vert X-Y\Vert_p.
	\end{equation}
\end{lemma}
\begin{proof} Since $Y$ is 
$\mathcal{H}\vee\mathcal{G}$-measurable,
\begin{eqnarray*}
E\left[ |X-E\left[ X \vert\mathcal{H}\vee\mathcal{G}\right]|^p\big\vert\mathcal{G}\right] &\leq&\\
2^{p-1}E\left[ |X-Y|^p\big\vert\mathcal{G}\right]
+ 2^{p-1}E\left[ |Y-E\left[ X\vert\mathcal{H}\vee\mathcal{G}\right]|^p\big\vert\mathcal{G}\right] &\leq&\\
2^{p-1}E\left[ |X-Y|^p\big\vert\mathcal{G}\right]
+ 2^{p-1}E\left[ E\left[ \vert Y-X\vert \big\vert\mathcal{H}\vee\mathcal{G}\right]^p\big\vert\mathcal{G}\right] &\leq&\\
2^{p-1}E\left[ |X-Y|^p\big\vert\mathcal{G}\right]
+ 2^{p-1}E\left[ E \left[  \vert Y-X\vert^p\big\vert\mathcal{H}\vee\mathcal{G} \right]\big\vert\mathcal{G} \right] &\leq&\\
2^{p}E\left[ \vert Y-X\vert^p\big\vert\mathcal{G}\right], & &
\end{eqnarray*}
by Jensen's inequality. Now \eqref{dirk} follows by taking $\mathcal{G}$ to be the trivial sigma-algebra. 
\end{proof}
We now note what happens to products of two random fields.
\begin{lemma}\label{product}
Let $X_t(\theta)$ be ULM-$rp$ and $Y_t(\theta)$ ULM-$rq$ where $r\geq 1$, $p>1$, $1/p+1/q=1$.
Then $X_t(\theta)Y_t(\theta)$ is ULM-$r$.
\end{lemma}
\begin{proof} We drop $\theta$ in the notation.
It is clear from H\"older's inequality that 
\[
M_r(XY)\leq M_{rp}(X)M_{rq}(Y),
\]
so $X_tY_t$ is bounded in $L^r$. Using Lemma \ref{mall}, let us estimate,
for $t,m\geq 1$,
\begin{eqnarray*}
	\Vert X_tY_t-E[X_tY_t\vert\mathcal{F}_{t-m}^+] \Vert_r  &\leq&\\
2\Vert X_tY_t-E[X_t\vert\mathcal{F}_{t-m}^+] E[Y_t\vert\mathcal{F}_{t-m}^+]\Vert_r &\leq&\\
2\Vert X_tY_t-X_t E[Y_t\vert\mathcal{F}_{t-m}^+]\Vert_r +
2\Vert X_t E[Y_t\vert\mathcal{F}_{t-m}^+]-E[X_t\vert\mathcal{F}_{t-m}^+] E[X_t\vert\mathcal{F}_{t-m}^+]\Vert_r
&\leq&\\ 2 \Vert X_t\Vert_{rp} \Vert Y_t-E[Y_t\vert\mathcal{F}_{t-m}^+]\Vert_{rq}+ 2
\Vert X_t- E[X_t\vert\mathcal{F}_{t-m}^+]\Vert_{rp}\Vert E[Y_t\vert\mathcal{F}_{t-m}^+]\Vert_{rq}&\leq&\\
2\Vert X_t\Vert_{rp} \Vert Y_t-E[Y_t\vert\mathcal{F}_{t-m}^+]\Vert_{rq}+
2\Vert X_t- E[X_t\vert\mathcal{F}_{t-m}^+]\Vert_{rp}\Vert Y_t\Vert_{rq},
\end{eqnarray*}
by H\"older's and Jensen's inequalities. This shows the $L$-mixing property
of order $r$, noting the assumptions on $X_t$, $Y_t$.
\end{proof}

\begin{lemma}\label{ken} Let $D$ be bounded.
	Fix $n\in\mathbb{N}$ and let $\psi_t$, $t\geq n$ be a sequence of $D$-valued, $\mathcal{F}_n$-measurable
	random variables. Let $X_t(\theta)$, $\theta\in D$,
$t\in\mathbb{N}$ be UCLM-$(p,1)$ for some $p>1$, satisfying the CLC property and
	define the process $Y_t:=X_{t}(\psi_t)$, $t\geq n$.
	Then
	$$
	M_p^{n}(Y)\leq M_p^{n}(X),\quad \Gamma_p^{n}(Y)\leq \Gamma_p^{n}(X)\mbox{ a.s.}
	$$
\end{lemma}
\begin{proof} If the $\psi_t$ are 
$\mathcal{F}_n$-measurable step functions then this follows easily
from the definitions. For general $\psi_t$, one can take $\mathcal{F}_n$-measurable 
step function approximations $\psi^k_t$, $k\in\mathbb{N}$ of the $\psi_t$ (in the almost sure sense). The CLC property
implies that $X_t(\psi^k_t)$ tends to $X_t(\psi_t)$ in probability as $k\to\infty$. 
By Fatou's lemma, $M_p^{n}(Y(\psi^k_{\cdot}))\leq M_p^{n}(X)$, $k\in\mathbb{N}$ now implies
$M_p^{n}(Y(\psi_{\cdot}))\leq M_p^{n}(X)$. The sequence $X_t(\psi^k_t)$ is bounded in $L^p$. 
It follows that $E[X_{n+t}(\psi^k_{n+t})\vert\mathcal{F}_{n+t-\tau}^+\vee\mathcal{F}_n]$ tends
to $E[X_{n+t}(\psi_{n+t})\vert\mathcal{F}_{n+t-\tau}^+\vee\mathcal{F}_n]$ in $L^1$, a fortiori, 
in probability. Hence, for each $\tau\geq 1$, 
$\gamma_p^{n}(Y(\psi_{\cdot}^k),\tau)\leq \gamma_p^{n}(X,\tau)$, $k\in\mathbb{N}$ implies 
$\gamma_p^{n}(Y(\psi_{\cdot}))\leq \gamma_p^{n}(X,\tau)$, by Fatou's lemma. Consequently,
$\Gamma_p^{n}(Y(\psi_{\cdot}))\leq \Gamma_p^{n}(X)$ a.s.
\end{proof}

\begin{remark}\label{minuse}
{\rm Fix $n\in\mathbb{N}$.
Let $Y_t$ be a conditionally $L$-mixing process of order $(p,1)$ for some $p\geq 1$ and define $W_t:=Y_t-E[Y_t\vert\mathcal{F}_n]$, $t\geq n$.
Then it is easy to check that $M^n_p(W)\leq 2M^n_p(Y)$ and $\Gamma_p^n(W)=\Gamma_p^n (Y)$.}
\end{remark}

Let us now enter the setting where for all $t\in\mathbb{N}$, $\mathcal{F}_t=\sigma(\varepsilon_j,\ j\in\mathbb{N},\ j\leq t)$, $\mathcal{F}_t^+:=\sigma(\varepsilon_j,\ j>t)$
for some i.i.d. sequence $\varepsilon_j$, $j\in\mathbb{Z}$ with values in some Polish space $\mathcal{X}$.
Let $\mu$ be the law of $(\varepsilon_0,\varepsilon_{-1},\ldots)$ on $\mathcal{X}^{-\mathbb{N}}$.
For given $\mathbf{e}=(e_0,e_{-1},\ldots)\in \mathcal{X}^{-\mathbb{N}}$
and $n\in\mathbb{N}$, we define the measure
$$
P^{\mathbf{e},n}:=\left(\otimes_{i>n} \nu\right) \bigotimes \left(\otimes_{i\leq n} \delta_{e_{i-n}}\right),
$$
where $\delta_x$ is the probability concentrated on the point $x\in\mathcal{X}$.
The corresponding expectation will be denoted by
$E^{\mathbf{e},n}[\cdot]$.

In this setting the concept of conditional $L$-mixing is easily related to ``ordinary'' $L$-mixing and we
will be able to use results of \cite{laci1} directly, see the proof of Theorem \ref{estim}.
For each $n\in\mathbb{Z}$, we denote by $Z_n$ the random variable $(\varepsilon_n,\varepsilon_{n-1},\ldots)$ and by $\tilde{\mu}$ their law
on $\mathcal{X}^{-\mathbb{N}}$ (which does not depend on $n$).
Let $X_t$, $t\in\mathbb{N}$ be a \emph{stochastic process} bounded in $L^r$ for some $r\geq 1$. We introduce the quantities
\begin{eqnarray*}
M^{\mathbf{e},n}_r(X) &:=& \sup_{t\in\mathbb{N}}
E^{\mathbf{e},n}[|X_{n+t}|^r]^{1/r},\\
\gamma^{\mathbf{e},n}_r(\tau,X) &:=& \sup_{t\geq\tau}
E^{\mathbf{e},n}[|X_{n+t}-E^{\mathbf{e},n}[X_{n+t}\vert \mathcal{F}_{n+t-\tau}^+]|^r]^{1/r},\ \tau\geq 1,\\
\Gamma^{\mathbf{e},n}_r(X) &:=&\sum_{\tau=1}^{\infty}\gamma^{\mathbf{e},n}_r(\tau,X),
\end{eqnarray*}
which are well-defined for $\tilde{\mu}$-almost every $\mathbf{e}$.

\begin{proof}[Proof of Theorem \ref{estim}.] 
For any non-negative random variable $Y$ on $(\Omega,\mathcal{F},P)$,
\begin{eqnarray}
E^{\mathbf{e},n}[Y]\Big\vert_{\mathbf{e}=Z_n} &=& E[Y\vert\mathcal{F}_n]\mbox{ a.s.}\label{putain}
\end{eqnarray}
This can easily be proved for indicators of the form
$Y=1_{\{\varepsilon_{n+j}\in A_j,\ -k\leq j\leq k\}}$ with some $k\in\mathbb{N}$ and
with Borel sets $A_j\subset\mathcal{X}$ and then it extends to all non-negative measurable $Y$.
It follows that
\begin{equation}\label{korker1}
M^n_r(W)=M^{\mathbf{e},n}_r(W)\vert_{\mathbf{e}=Z_n}. 
\end{equation}
A similar argument also establishes
\begin{eqnarray}\label{oregbacsi}
E^{\mathbf{e},n}[Y\vert\mathcal{F}_{n+t-\tau}^+]\Big\vert_{\mathbf{e}=Z_n} &=& E[Y\vert\mathcal{F}_{n+t-\tau}^+\vee\mathcal{F}_n]\mbox{ a.s.},
\end{eqnarray}
for all $t\geq 1$ and $1\leq \tau\leq t$ hence also
\begin{equation}\label{korker2}
\gamma^{\mathbf{e},n}_r(\tau,X)\vert_{\mathbf{e}=Z_n}=\gamma^n_r(\tau,X)\mbox{ a.s.} 
\end{equation}
From the conditional $L$-mixing property of $W_t$, $t\in\mathbb{N}$ under $P$ (of order $(r,1)$)
it follows that, for $\tilde{\mu}$-almost every $\mathbf{e}$, 
the process $W_{t+n}$, $t\in\mathbb{N}$ is $L$-mixing under $P^{\mathbf{e},n}$.
Theorems 1.1 and 5.1 of \cite{laci1} (applied under $P^{\mathbf{e},n}$) imply
$$
E^{\mathbf{e},n}\left[ \max_{n < t \le m} 
\left| \sum_{s = n+1}^{t} b_s W_s\right|^r  \right]^{1/r} 
\le C_r \left( \sum_{s=n+1}^{m} b_s^2 \right)^{1/2} \sqrt{{M}^{\mathbf{e},n}_r(W) \Gamma_r^{\mathbf{e},n}(W)},$$
for $\tilde{\mu}$-almost every $\mathbf{e}$.
Now \eqref{putain}, \eqref{korker1} and \eqref{korker2} imply \eqref{mandrill}.
\end{proof}

Now we turn to the proofs of Section \ref{harom}.
We first recall Lemma 2.2 of \cite{laci3}, which states that the discrete flow defined by (\ref{discrete}) below inherits the exponential
stability property \eqref{gag}.
Let $\mathbb{M}:=\{(m,n)\in\mathbb{N}:\ m\leq n\}$.

\begin{lemma}\label{expstab} Let Assumptions \ref{boun} and \ref{stab} be in force.
For each $0\leq m\leq n$ and $\xi\in D_{\xi}$, define $z:\mathbb{M}\times D\to D$ 
by the recursion
\begin{equation}\label{discrete}
z(m,m,\xi):=\xi,\quad z(n+1,m,\xi):=z(n,m,\xi)+\lambda G(z(n,m,\xi)).
\end{equation}
If $d$ is large enough and $\lambda$ is small enough then this makes sense and $z(n,m,\xi)\in D_{\theta}$ for all $n\geq m$.
Furthermore, for each $\alpha'<\alpha$ (see Assumption \ref{stab}) there is $C(\alpha')>0$ such that
\begin{equation}\label{tok}
\left\vert \frac{\partial}{\partial \xi}z(n,m,\xi)\right\vert\leq C(\alpha') e^{-\lambda\alpha'(n-m)}.
\end{equation}
\hfill$\Box$\end{lemma}

\begin{remark}\label{kellmajd}
{\rm Actually, the same arguments also imply that the recursion \eqref{discrete}
is well-defined for all $\xi\in D_{\theta}$,
stays in $D$ and satisfies \eqref{tok}, provided that $d'$ is large enough and $\lambda$ is
sufficiently small.}
\end{remark}

For convenience's sake, we recall a result from \cite{geman}, which is also given as Lemma 4.2 of \cite{laci3}.
\begin{lemma}\label{lem_dif}
Let Assumptions \ref{boun} and \ref{stab} be satisfied. Let $y_t:=y(t,0,\xi)$, $t\geq 0$. 
Let $x_t$, $t\geq 0$ be a continuous, piecewise continuously differentiable curve such that $x_0 = \xi$. Then for $t\geq 0$,
\begin{equation}\label{dif_ode}
x_t - y_t = \int_0^t{ \frac{\partial}{\partial \xi}y(t,w,x_w)(\dot{x}_w - G(x_w))dw }.
\end{equation} 
\end{lemma}
\begin{proof} 
For $0 \le w \le t$, let $z_w = y(t,w,x_w)$. The LHS of (\ref{dif_ode}) can be written as 
$$z_t - z_0 = \int_0^t{\dot{z}_wdw} = \int_0^t{\left( \frac{\partial}{\partial w} y(t,w,x_w) + \frac{\partial}{\partial \xi}y(t,w,x_w)\dot{x}_w\right)dw}.$$
From Theorem 3.1 on page 96 of \cite{hartman} we obtain that, for all $x\in\mathbb{R}$,
$$\frac{\partial}{\partial w}y(t,w,x) + \frac{\partial}{\partial \xi}y(t,w,x)G(x) = 0,$$
and hence the proof is complete.
\end{proof}

Let $\xi\in D_{\theta}$ and define $\tilde{z}_n:=z(n,0,\xi)$, $n\in\mathbb{N}$.
The next lemma summarizes some arguments of \cite{laci3} in the present setting, for the sake
of a self-contained presentation.

\begin{lemma}\label{matti}
Let Assumptions \ref{boun} and \ref{stab} be satisfied. 
Let $y_t:=y(t,0,\xi)$ for some $\xi\in D_{\xi}$ and let $\theta_n$ be defined by \eqref{bab}.  
If $d, d'$ are large enough then, for all $n \in \mathbb{N}$, we have $\theta_n \in D_{\theta}$ 
and also $\tilde{z}_n \in D$.
\end{lemma}
\begin{proof}
We denote by $\theta_t$ the piecewise linear extension of $\theta_n$, i.e. for $t \in (n,n+1)$, 
we set $\theta_t = (1-(t-n))\theta_n + (t-n)\theta_{n+1}.$ For $w \in (n,n+1)$, it is easy to see that
$\dot{\theta}_w = \theta_{n+1} - \theta_n= \lambda H(\theta_{[w]}, X_{[w]+1}) $ where $[w]$ denotes the integer part of $w$. 
Thus, Lemma \ref{lem_dif} implies that as long as $\theta_w \in D_{\theta}$ for all $0 \le w \le t$,
$$\theta_t - y_t = \int_0^t{ \frac{\partial}{\partial \xi} y(t,w,\theta_w) \lambda\left(H(\theta_{[w]},X_{[w]+1}) - G(\theta_w) 
\right) dw}.$$
Since $|H|$ and $|G|$ are bounded by a constant, say, $C^{\dagger}$, (\ref{gag}) implies that
$$|\theta_t - y_t| \le \int_0^t{C^* e^{-\lambda \alpha (t-w)} \lambda 2C^{\dagger} dw}\le 
2C^*C^{\dagger}\alpha^{-1}.$$
It is known that $y_t \in D_y$ whenever $y_0 \in D_{\xi}$. Now, if 
$d>2C^*C^{\dagger}\alpha^{-1}$ then $|\theta_t - y_t|$ will be smaller than the distance between 
$D_{y}$ and $D_{\theta}^c$, where $D_{\theta}^c$ denotes the complement of $D_{\theta}$, 
hence $\theta_t$ will stay in $D_{\theta}$ for ever.

The proof for $\tilde{z}_n \in D$ is similar. 
The piecewise linear extension of $\tilde{z}_n$ is denoted by $\tilde{z}_t$, $t\geq 0$.
By computations as before,
$$
\tilde{z}_t - y_t = \int_0^t{\frac{\partial}{\partial \xi} y(t,w,\tilde{z}_w) 
\lambda\left(G(\tilde{z}_{[w]}) - G(\tilde{z}_w) \right)  dw}.
$$
Denoting by $K^*$ (resp. $L^*$) a bound for $|G|$ (resp. a Lipschitz-constant for $G$), we obtain 
$$
|G(\tilde{z}_{[w]}) - G(\tilde{z}_w)|\le L^*|\tilde{z}_{[w]} - \tilde{z}_w|\le L^* \lambda G(\tilde{z}_{[w]}) 
\le \lambda L^*K^*,
$$
hence
$$|\tilde{z}_t - y_t| \le \int_0^t{ C^* e^{-\lambda \alpha (t-w)} \lambda^2 L^*K^* dw } \le C^* \alpha^{-1} \lambda L^*K^*.$$
It follows that if $d' >C^*\alpha^{-1} \lambda L^*K^*$ then $\tilde{z}_t \in D$, for all $t$.  
\end{proof}

\begin{remark}{\rm Note that our estimates for $d$, $d'$ in the above proof are
somewhat different: by choosing $\lambda$ small enough we can make $d'$ as small
as we wish whereas we do not have this option for $d$. This is in contrast with
\cite{laci3}, where $d$ can also be made arbitrarily small by choosing $\lambda$
small. This difference comes from the fact that in \cite{laci3} Lipschitz-continuity
of $\theta\to H(\theta,\cdot)$ is assumed, unlike in the present setting.}
\end{remark}

\begin{proof}[Proof of Theorem \ref{main}.] We follow the main lines of the arguments in \cite{laci2,laci3}.
However, details deviate significantly as our present assumptions are different from those of
the cited papers.

Lemma \ref{matti} above will guarantee that $\theta_t$ and $z_t,\overline{z}_t$ (see below) are well-defined.
Clearly, $z_t = z(t,0,\theta_0)$. Set $T = [1/(\lambda\alpha')]$, where $0<\alpha'<\alpha$ 
is as in Lemma \ref{expstab}
and $[x]$ denotes the integer part of $x\in\mathbb{R}$.
For each $n\in\mathbb{N}$, we set $\overline{z}_{nT}:=\theta_{nT}$ and define recursively
\begin{align*}
\overline{z}_t:&=\overline{z}_{t-1}+\lambda G(\overline{z}_{t-1}), \qquad  nT < t < (n+1)T.
\end{align*}
In other words, $\overline{z}_t=z(t,nT,\theta_{nT})$.
By the triangle inequality, we obtain, for any $t \in \mathbb{N} $,
\begin{equation}\label{triangle}
|\theta_t - z_t| \le |\theta_t - \overline{z}_t| + |\overline{z}_t - z_t|.
\end{equation}
\textsl{Estimation for $|\theta_t - \overline{z}_t|$.} Fix $n$ and let $nT<t<(n+1)T$.
\begin{eqnarray*}
|\theta_t-\overline{z}_t|=\lambda \left\vert \sum_{k=nT}^{t-1} [H(\theta_k,X_{k+1})-G(\overline{z}_k)]\right\vert &\leq&\\
\lambda\sum_{k=nT}^{t-1} \left\vert H(\theta_k,X_{k+1})-H(\overline{z}_k,X_{k+1})\right\vert &+&\\
\lambda\left\vert\sum_{k=nT}^{t-1} \left(H(\overline{z}_k,X_{k+1})-E[H(\overline{z}_k,X_{k+1})\vert\mathcal{F}_{nT}]\right)
\right\vert &+&\\
\lambda \sum_{k=nT}^{t-1}
\left\vert E[H(\overline{z}_k,X_{k+1})\vert\mathcal{F}_{nT}]-G(\overline{z}_k)\right| &=:& \lambda(S_1+S_2+S_3).
\end{eqnarray*}
It is clear that
\begin{eqnarray*}
ES_3 &\leq&
E\left[\sup_{\vartheta\in D}\sum_{k=nT}^{\infty} \left| E[H(\vartheta,X_{k+1})\vert\mathcal{F}_{nT}] -G(\vartheta) \right|\right]<C',
\end{eqnarray*}
for some $C'<\infty$, by Assumption \ref{infsum}.

Turning our attention to $S_1$, the CLC property implies
$$
ES_1=E[E[S_1\vert\mathcal{F}_{nT}]]\leq \sum_{k=nT}^{t-1} K E|\theta_k-\overline{z}_k|.
$$
On each interval $nT\le t < (n+1)T$, we now estimate $S_2$ as follows,
$$
S_2\leq \sup_{nT<t\leq (n+1)T} \left|\sum_{k=nT}^{t-1}
\left(H(\overline{z}_k,X_{k+1})-E[H(\overline{z}_k,X_{k+1})\vert\mathcal{F}_{nT}]\right)\right|.
$$
Note the UCLM-$(r,1)$ property of $H(\cdot,\cdot)$ as well as Lemma
\ref{ken} and Remark \ref{minuse}. Apply Theorem
\ref{estim} for $nT$ instead of $n$ and with the choice 
$b_t\equiv 1$ and $$
W_t:=H(\overline{z}_t,X_{t+1})-E[H(\overline{z}_t,X_{t+1})\vert\mathcal{F}_{nT}],\ nT<t\leq (n+1)T,\ W_t:=0,\ 0\leq t\leq nT,
$$
note that $E[W_t\vert\mathcal{F}_{nT}]=0$ for all $t$.
We get 
\begin{eqnarray*}
ES_2=E[E[S_2\vert\mathcal{F}_{nT}]]\leq E[E^{1/r}[S_2^r\vert\mathcal{F}_{nT}]] &\leq&\\
C_r T^{1/2} E\left[\sqrt{M^{nT}_r(W)\Gamma^{nT}_r(W)}\right]\leq C_r T^{1/2} 
\sqrt{EM^{nT}_r(W) E\Gamma^{nT}_r(W)} &\leq&\\
C''T^{1/2} & &
\end{eqnarray*}
with some $C''<\infty$, independent of $n$, by the UCLM-$(r,1)$ property of $W$.

Putting together our estimates so far, we obtain for $nT\le t < (n+1)T$,
$$E|\theta_t - \overline{z}_t| \leq \lambda\left( \sum_{k=nT}^{t-1}K E|\theta_k-\overline{z}_k| +
C''T^{1/2} + C'\right).$$
Recall that $E|\theta_t - \overline{z}_t|$ is finite by boundedness of $D$. The discrete Gronwall lemma yields the following estimate, independent of $n$:
\begin{equation}\label{sofronitsky}
E|\theta_t - \overline{z}_t| \leq \lambda(C''T^{1/2}+C')(1+\lambda K)^{T}.
\end{equation}
Note that
$$
(1+\lambda K)^T\leq e^{\lambda KT}\leq e^{K/\alpha'}.
$$

\noindent\textsl{Estimation for $|\overline{z}_t - z_t|$.} Noting $z_0 = \theta_0$ and using the fundamental theorem of calculus, 
we estimate for $nT \leq t < (n+1)T$, using telescoping sums,
\begin{eqnarray*}
& & |\bar{z}_t - z_t| \\
&\leq& \sum_{k=1}^n |z(t,kT,\theta_{kT}) - z(t, (k-1)T, \theta_{(k-1)T})| \\
&=& \sum_{k=1}^n |z(t,kT,\theta_{kT}) - z(t, kT, z(kT,(k-1)T, \theta_{(k-1)T}))| \\
&=& 
\sum_{k = 1}^{n} \int_0^1 \left| \frac{\partial}{\partial \xi}z(t,kT, s \theta_{kT} + (1-s)z(kT,(k-1)T, \theta_{(k-1)T})) 
\right|  ds  \\
&\times& |\theta_{kT} - z(kT,(k-1)T, \theta_{(k-1)T})|\\
&\leq&  C(\alpha') \sum_{k = 1}^{n}  e^{-\lambda \alpha'(t-kT)} \left( |\theta_{kT -1} - \bar{z}_{kT-1}| + 
\lambda |H(\theta_{kT-1},X_{kT}) - G(\overline{z}_{kT-1})|\right).
\end{eqnarray*}
Notice that there is $\tilde{C}>0$, independent of $n,t$ such that
\[
\sum_{k=1}^n e^{-\lambda\alpha'(t-kT)}\leq \tilde{C}.
\]
Therefore, the fact that $H$, $G$, $D$ are bounded, imply
\begin{eqnarray}\nonumber
E|\overline{z}_t - z_t| &\leq& c \sum_{k = 1}^{n} e^{-\lambda \alpha'(t-kT)} E|\theta_{kT -1} - \bar{z}_{kT-1}|+ c \sum_{k = 1}^{n} e^{-\lambda \alpha'(t-kT)} \lambda \\
&\leq& c' \lambda^{1/2},\label{pattoo}
\end{eqnarray}
with some $c,c'>0$, by \eqref{sofronitsky} and by the choice of $T$.
Finally, putting together our estimations \eqref{sofronitsky}, \eqref{pattoo} and using (\ref{triangle}), for $\lambda$ small
enough, we obtain
$$E|\theta_t - z_t| \leq C \lambda^{1/2},$$
with some $C>0$, which completes the proof.
\end{proof}

\begin{proof}[Proof of Corollary \ref{karszt}.]
Recall $\alpha'$ from  Lemma \ref{expstab}. The fundamental theorem of calculus yields
\begin{align*}
|z_t - \theta^*| &\leq |z_0 - \theta^*|\int_0^1 \left| \frac{\partial}{\partial \xi}z(t,0, s z_0 + (1-s)\theta^*) \right|  ds\\
&\le C(\alpha') e^{-\lambda \alpha' t} |z_0 - \theta^*|,
\end{align*}
and this is $\leq \lambda^{1/2}$ for $t\geq t_0(\lambda)$ if $t_0(\lambda)=C^{\circ}\ln(1/\lambda)/\lambda$ for some $C^{\circ}$.
Since
$$
|\theta_t-\theta^*|\leq |\theta_t-z_t|+|z_t - \theta^*|,
$$
the statement follows.
\end{proof}

\begin{proof}[Proof of Theorem \ref{louvre}.] Let us work conditionally on the event 
$\mathcal{E}_0=\eta\in \mathbb{R}^{\mathbb{Z}}$ where 
$$
\mathcal{E}_l=(\varepsilon^1_{i+l})_{i\in\mathbb{Z}},
$$
until further notice.

The CLC property and Assumption \ref{boun} are trivial. Define $\mathcal{F}_n:=\sigma(\varepsilon^2_j; j\leq n)$ and  
$\mathcal{F}^+_n:=\sigma(\varepsilon^2_j; j> n)$. 

We now prove that $H(\theta,X_t)$ is UCLM-$(r,1)$ with respect to the
given $(\mathcal{F}_n,\mathcal{F}^+_n)$. Boundedness of $H$ implies that 
$M^n_r(X)$, $n\in\mathbb{N}$ is uniformly bounded.

Fix $1\leq m\leq t$. Define recursively
$$
\xi_{t-m}:=x,\quad \xi_{l+1}:=F(\mathcal{E}_{l+1},\varepsilon_{l+1}^2,\xi_l),\ l\geq t-m.
$$

Set $X_{t,m}^+:=\xi_t$. By construction, $X_{t,m}^+$ is $\mathcal{F}_{t-m}^+$-measurable and
$$
\vert H(\theta,X_{t,m}^+)-H(\theta,X_t)\vert \leq L\rho^{m} |x-X_{t-m}|,
$$
where $L$ is a Lipschitz-constant for $x\to H(\theta,x)$. 
So we can further estimate 
\begin{eqnarray*}
E\left[ |x-X_{t-m}|^r\big\vert\mathcal{F}_0\right]^{1/r} \leq \sum_{j=1}^{\infty} 
E\left[ |\tilde{X}^{t-m}_j-\tilde{X}^{t-m}_{j-1}|^r\big\vert\mathcal{F}_0\right]^{1/r} &\leq&\\ 
\sum_{j=1}^{\infty} \rho^{j-1} E\left[ |x- F(\mathbf{f}_{t-m-j+1},\varepsilon_{t-m-j+1}^2,x)\vert^r 
\big\vert\mathcal{F}_0\right]^{1/r} &\leq&\\ 
C\sum_{j=1}^{t-m} \rho^{j-1} \Vert \vert\varepsilon_{t-m-j+1}^2\vert+1\Vert_r
+C\sum_{j=t-m+1}^{\infty} \rho^{j-1} [\vert\varepsilon_{t-m-j+1}^2\vert+1] &\leq&\\
C\Vert \vert\varepsilon_{0}^2\vert+1\Vert_r\sum_{j=1}^{\infty} \rho^{j-1} 
+C\sum_{k=0}^{\infty} \rho^{t-m+k} [\vert\varepsilon_{-k}^2\vert+1] &\leq&\\
C\Vert \vert\varepsilon_{0}^2\vert+1\Vert_r\sum_{j=1}^{\infty} \rho^{j-1} 
+C\sum_{k=0}^{\infty} \rho^{k} [\vert\varepsilon_{-k}^2\vert+1],
\end{eqnarray*}
using Assumption \ref{contract}, the independence of $\varepsilon_j^2$, $j\geq 1$ from $\mathcal{F}_0$ and
the $\mathcal{F}_0$-measurability of $\varepsilon^2_j$, $j\leq 0$. Note that this last estimate is
independent of $t$.
We can carry out analogous estimates with $\mathcal{F}_n$ instead of $\mathcal{F}_0$ and these imply, via Lemma \ref{mall}, 
$$
\gamma_r^n(m,X)\leq 2LC\rho^{m}\left[(\Vert \varepsilon_0^2\Vert_r+1)\sum_{j=1}^{\infty} \rho^{j-1}
+\sum_{k=0}^{\infty} \rho^{k} [\vert\varepsilon_{n-k}^2\vert+1]\right],
$$
for each $n\in\mathbb{N}$, which implies that the sequence $\Gamma^n_r(X)$ is bounded in $L^1$,
showing the UCLM-$(r,1)$ property for $H(\theta,X_t)$. 

Since $X_{t-m}^+$ is $\mathcal{F}_{t-m}^+$-measurable, the above estimates also show that 
$H(\theta,X_t)$ is (unconditionally) $L$-mixing of order $(r,1)$, hence 
Remark \ref{frater} implies 
Assumption \ref{infsum}. As the estimates are uniform in $\eta\in\mathbb{R}^{\mathbb{Z}}$,
the argument of Theorem \ref{main} can be applied. 
\end{proof}

\section{Conclusion}

There is a large number of natural ramifications of our results that could be pursued: the estimation of higher
order moments of the tracking error using the property UCLM-$(r,p)$ for $p>1$; accommodating multiple roots for equation \eqref{g}; proving the
convergence of the decreasing gain version of \eqref{elso}; considering the convergence
of concrete procedures. We leave these for later work in order to
convey a clear message, highlighting the novel techniques we have introduced.

\medskip

\noindent\textbf{Acknowledgments.} We thank two anonymous referees for
several insightful comments that led to substantial improvements.
The major part of this work was done while the second author was working as a Whittaker Research Fellow in Stochastic Analysis in the School of Mathematics, University of Edinburgh, United Kingdom. 
We have made use of the resources provided by the Edinburgh Compute and Data Facility (ECDF), see
$$
\mathtt{http://www.ecdf.ed.ac.uk/}
$$
This work was supported by The Alan Turing Institute under the EPSRC grant EP/N510129/1, in the
framework of a ``small research group'', during the summer of 2016. 
Huy N. Chau and Mikl\'os R\'asonyi were also supported by the
``Lend\"ulet'' Grant LP2015-6 of the Hungarian Academy of Sciences and
by the NKFIH (National Research, Development and Innovation Office, Hungary) 
grant KH 126505. 
Sotirios Sabanis gratefully acknowledges the support of the Royal Society through the IE150128 grant.
We thank L\'aszl\'o Gerencs\'er for helpful discussions and dedicate this paper to him.

\end{document}